\documentclass[10pt,reqno]{amsart}

\usepackage{amsmath,amssymb,mathtools}
\usepackage[left=2.6cm,right=2.6cm,top=2.2cm,bottom=2.4cm]{geometry}
\usepackage{microtype}
\usepackage[colorlinks=true,urlcolor=blue,citecolor=red,linkcolor=blue,
linktocpage,pdfpagelabels,bookmarksnumbered,bookmarksopen]{hyperref}

\allowdisplaybreaks
\numberwithin{equation}{section}

\newtheorem{thm}{Theorem}[section]
\newtheorem{lem}[thm]{Lemma}
\newtheorem{cor}[thm]{Corollary}

\theoremstyle{definition}

\theoremstyle{remark}
\newtheorem{remark}[thm]{Remark}

\newcommand{\RR}{\mathbb R}

\begin{document}
	
	\title[Rupture solutions of a biharmonic equation]
	{Rigidity of positive rupture solutions to a biharmonic equation with critical negative exponent}
	
	\author{Xia Huang}
	\address{School of Mathematical Sciences, Key Laboratory of MEA
	(Ministry of Education) \& Shanghai Key Laboratory of PMMP,
	East China Normal University}
	\email{xhuang@cpde.ecnu.edu.cn}

	\author{Yahui Jiang}
	\address{School of Mathematical Sciences, Key Laboratory of MEA
	(Ministry of Education) \& Shanghai Key Laboratory of PMMP,
	East China Normal University}
	\email{52275500028@stu.ecnu.edu.cn}

	\author{Xianmei Zhou}
	\address{School of Mathematical Sciences, Zhejiang Normal University}
	\email{xmzhou@zjnu.edu.cn}
	
	
	\begin{abstract}
We establish two rigidity theorems for positive rupture solutions of the conformally invariant equation
$\Delta^2u=u^{-7}$ in $\mathbb R^3\setminus\{0\}$, which extend continuously
to the origin with $u(0)=0$. First, we prove that if the associated conformal metric $g=u^{-4}|dx|^2$ has nonnegative scalar
curvature, then every such solution is radially symmetric and has the sharp rupture profile
$u(x)\sim(4/3)^{1/4}|x|^{1/2}$ as $x\to 0$; in particular, the metric is
complete at the origin. The principal novelty is a global rigidity theorem requiring neither
curvature nor symmetry: the single global condition $u(x)=o(|x|)$ at infinity
forces $u(x)\equiv(4/3)^{1/4}|x|^{1/2}$. This result provides a sharp answer to the uniqueness question posed by
McKenna and Reichel \cite{McKennaReichel} in a class with no a priori symmetry assumption. Finally, we construct complementary examples demonstrating the essential roles
of the curvature and growth hypotheses.
\end{abstract}

\maketitle

\section{Introduction}

Conformally covariant fourth-order equations form a natural bridge between
nonlinear elliptic PDEs and conformal geometry. The
Graham--Jenne--Mason--Sparling family \cite{GJMS} contains the Paneitz operator
\cite{Branson,Paneitz}; in dimension four, its transformation law governs the
$Q$-curvature and the associated variational problem \cite{ChangYang}. When $n\neq4$, write
$g_{ij}=u^{4/(n-4)}\delta_{ij}$ and normalize the $Q$-curvature so that the
conformal factor satisfies
\[
\Delta^2u=\frac{n-4}{2}Q_g u^{\frac{n+4}{n-4}}.
\]
When $n=3$, the conformal exponent is negative, so a vanishing conformal
factor produces a singular nonlinearity. Taking $Q_g=-2$ and prescribing a
zero at the origin, we are led to the problem
\begin{equation}
\label{eq1}
\begin{cases}
\Delta^2 u=u^{-7} & \text{in }\RR^3\setminus\{0\},\\
u(x)\longrightarrow0 & \text{as }x\longrightarrow0.
\end{cases}
\end{equation}

The second condition in \eqref{eq1} allows every positive solution on the
punctured space to be extended continuously to all of $\mathbb R^3$ by
defining $u(0)=0$. We call such a solution a \emph{rupture solution}. Although
the extended function vanishes at the origin, the nonlinearity $u^{-7}$ blows
up there; hence the puncture remains nonremovable for the equation. The
conformal metric
\[
g=u^{-4}|dx|^2
\]
is likewise singular at the origin. Its completeness and curvature near the
puncture are therefore encoded in the precise asymptotic behavior of $u$.
The scalar curvature is given by
\begin{equation}
\label{eq:scalar-curvature-identity}
R_g=8u^2\bigl(u\Delta u-2|\nabla u|^2\bigr).
\end{equation}

\medskip

Equation \eqref{eq1} is invariant under the scaling
$u_\rho(x)=\rho^{-1/2}u(\rho x)$ and under the Kelvin transformation
$\widehat u(x)=|x|u(x/|x|^2)$. The homogeneous ansatz
$u(x)=c|x|^{1/2}$ gives
\[
\Delta^2(c|x|^{1/2})=\frac9{16}c|x|^{-7/2},
\qquad
(c|x|^{1/2})^{-7}=c^{-7}|x|^{-7/2}.
\]
Thus $c^8=16/9$, and the canonical rupture solution is
\begin{equation}
\label{eq:explicit-solution}
u_*(x)=c_*|x|^{1/2},
\qquad c_*:=\left(\frac{4}{3}\right)^{1/4}.
\end{equation}
Three-dimensional fourth-order conformal equations and their geometric foundations were studied in \cite{XuYang}, while exact solutions of the associated integral equations were investigated in \cite{Xu}.
 The central questions are whether geometric information
forces an arbitrary rupture solution to approach $u_*$ at the puncture and
whether a condition imposed only at infinity can force equality everywhere.

\medskip

The sign in $\Delta^2u=\pm u^{-q}$ has a decisive effect on the solution
theory. The negative-sign equation $\Delta^2u=-u^{-q}$ appears in models for
MEMS and thin films; see \cite{GazzolaGrunauSweers2010,PeleskoBernstein}.
Choi and Xu \cite{ChoiXu} and Guerra \cite{Guerra} studied radial existence
and growth, Hyder and Wei \cite{HyderWei} constructed nonradial solutions and
pointwise inequalities and symmetry criteria were developed in
\cite{GuoWeiZhou,NgoNguyenPhan}. These works concern the opposite sign and,
for the most part, solutions that remain positive throughout $\RR^3$.

\medskip

The positive-sign equation $\Delta^2u=u^{-q}$ has a different endpoint
structure. In their pioneering study of the radial problem,
McKenna and Reichel \cite{McKennaReichel} asked in Section~6, Question~(6),
for which classes of solutions the model $c_*|x|^{1/2}$ is unique. Their
question is especially delicate without radial symmetry: the singular
nonlinearity prevents uniform elliptic control near the rupture point, while
the biharmonic operator has no general maximum principle with which to compare
two solutions. Theorem~\ref{thm2} answers the question in the class defined by
sublinear growth at infinity. The class is genuinely nonradial, and the
conclusion recovers both symmetry and the exact profile.

\medskip

For comparison, the higher-dimensional critical-power equation
\[
\Delta^2u=u^{\frac{n+4}{n-4}}
\qquad\text{in }\mathbb R^n\setminus\{0\},\qquad n\geq5,
\]
has a markedly different singular behavior. Positive entire solutions are
classified by the standard bubbles \cite{Lin,WeiXu}. Positive solutions with
a nonremovable isolated singularity are of Delaunay type: after rescaling,
$|x|^{(n-4)/2}u(x)$ is periodic in $\log|x|$, so the solution oscillates on
logarithmic scales rather than approaching a single profile; see
\cite{FrankKonig,GuoHuangWangWei}. The analogous second-order theory for the
Yamabe equation includes asymptotic symmetry and refined Fowler-type
asymptotics; see \cite{CGS,KMPS,Marques,MazzeoPacard,XiongZhang}.

\medskip

Our proofs are driven by two structural difficulties. First, at the rupture point, $u$ vanishes and the negative power nonlinearity $u^{-7}$ blows up,
so scale-invariant lower bounds must precede any compactness argument. Second,
the biharmonic operator has no general maximum principle, rendering the classical moving-plane
method ineffective. We overcome the first obstacle by extracting the second-order
inequality
\[
u\Delta u\geq2|\nabla u|^2
\]
from nonnegative scalar curvature. To overcome the second, we pass to the
cylinder and use coercive identities that force invariance under translation and
rotation.

\medskip

Our first theorem shows that the curvature inequality alone supplies enough
local and global control to determine both symmetry and the leading rupture
profile.

\begin{thm}
\label{thm1}
Let $u\in C^4(\RR^3\setminus\{0\})$ be a positive solution of \eqref{eq1}.
If the metric $g=u^{-4}|dx|^2$ has $R_g\geq0$, then $u$ is radially symmetric
about the origin and
\begin{equation}
\label{eq:main-origin-asymptotic}
\lim_{r\downarrow0}\frac{u(r)}{\sqrt r}=c_*.
\end{equation}
In particular, $g$ is complete at the origin.
\end{thm}

The proof converts the curvature sign into scale-invariant estimates, a global
growth bound, and control of the Kelvin transform. A rotation-difference
identity then eliminates every angular mode. Once radiality is known, an
autonomous fourth-order equation on the cylinder identifies the unique
negative-end limit. The asymptotic formula also implies the logarithmic
divergence of the length of every curve approaching the puncture.

\begin{remark}
\label{rem:curvature-assumption}
The curvature hypothesis in Theorem~\ref{thm1} is substantive. There exists a
nonradial positive solution in $C^\infty(\RR^3\setminus\{0\})$ whose scalar
curvature changes sign; see
Lemma~\ref{prop:variational-nonradial-counterexample}.
\end{remark}

Our second theorem is the principal rigidity result. It replaces the
differential curvature condition by information at only one end of the
punctured space. The hypothesis gives no a priori angular control and does not
place the solution in the radial class studied previously.

\begin{thm}
\label{thm2}
Let $u\in C^4(\RR^3\setminus\{0\})$ be a positive solution of \eqref{eq1}.
If
\begin{equation}
\label{eq:apsc-infinity-assumption}
u(x)=o(|x|)\qquad\text{as }|x|\to\infty,
\end{equation}
uniformly in direction, then
\begin{equation}
\label{eq:apsc-classification}
u(x)=c_*|x|^{1/2}\qquad\text{in }\RR^3\setminus\{0\}.
\end{equation}
Consequently, $R_g\equiv8/3$.
\end{thm}

The key device is a weighted two-solution uniqueness principle on
$\mathbb R\times\mathbb S^2$. Translation in the cylindrical variable
corresponds to scaling in Euclidean space, while rotations act on the sphere.
We compare the cylindrical profile simultaneously with all of its translates
and rotations. The weight $\cosh(t/2)$ matches the largest endpoint growth
allowed by the rupture and sublinear-growth assumptions; after conjugation,
it reveals a strictly positive zeroth-order term. This converts endpoint
smallness into a global coercive identity and forces invariance under both
groups. Thus Theorem~\ref{thm2} provides a sharp answer to Question~(6) of McKenna and Reichel 
\cite[Section~6, Question~(6)]{McKennaReichel}: the canonical solution is unique under a single
global growth condition, with symmetry obtained as part of the conclusion.

\begin{remark}
\label{rem:sublinear-growth-essential}
The growth hypothesis in Theorem~\ref{thm2} cannot be omitted altogether.
Indeed, Corollary~\ref{cor:radial-linear-growth} constructs positive radial
rupture solutions with nonzero linear growth at infinity that are distinct
from the canonical solution \(u_*\). Thus the sublinear condition separates
the rigidity class from a family of noncanonical solutions.
\end{remark}

The paper is organized as follows. Section~2 derives the curvature-based
estimates, proves radial symmetry, and identifies the rupture profile in
Theorem~\ref{thm1}. Section~3 develops the weighted two-solution principle and
proves Theorem~\ref{thm2}. Section~4 gives variational constructions that test
the necessity of the two rigidity hypotheses.

\section{A priori estimates and radial symmetry}\label{sec:symmetry}

The proof of Theorem~\ref{thm1} proceeds in three stages. We first derive
a local estimate under the scalar-curvature condition. A
blow-up argument and a Liouville obstruction then yield a scale-invariant
lower bound and a spherical Harnack inequality. These estimates control both
ends of the cylindrical profile. A cutoff identity for rotational derivatives
then yields radial symmetry, and an autonomous ODE identifies the limiting constant.

\begin{lem}\label{lem:local-curvature-estimate}
Let \(U\in C^4(B_4)\) be positive and suppose that
	\begin{equation}\label{eq:local-assumptions}
		U\Delta U\geq 2|\nabla U|^2
		\quad\text{and}\quad
		0\leq\Delta^2U\leq F
		\quad\text{in }B_4,
	\end{equation}
	where \(F\geq0\) is constant. Then there is a universal constant \(C>0\)
	such that 
	\begin{equation}\label{eq:local-curvature-estimate}
		\Delta U(0)+\sup_{B_1}U
		\leq C\bigl(U(0)+F\bigr).
	\end{equation}
\end{lem}
\begin{proof}
	Set \(w:=\Delta U\) and \(f:=\Delta w=\Delta^2U\). Since \(U>0\), the
	first inequality in \eqref{eq:local-assumptions} gives \(w\geq0\). Thus
	\(0\leq\Delta w=f\leq F\) in \(B_4\).

\medskip

We first derive two elementary local estimates. Let \(p\) solve
	\begin{equation*}
		\Delta p=f\quad\text{in }B_{7/2},
		\qquad p=0\quad\text{on }\partial B_{7/2},
	\end{equation*}
The maximum principle gives \(p\leq0\). For the reverse bound, set
	\begin{equation*}
		\phi(x)=\frac{F}{6}\left(|x|^2-\frac{49}{4}\right).
	\end{equation*}
Then \(\Delta\phi=F\) in \(B_{7/2}\) and \(\phi=0\) on
	\(\partial B_{7/2}\). Since \(\Delta(p-\phi)=f-F\leq0\) in \(B_{7/2}\), the minimum
	principle yields \(p\geq\phi\). Hence
	\(-\frac{49}{24}F\leq p\leq0\) in \(B_{7/2}\).

	Now \(h:=w-p\geq0\) is harmonic in \(B_{7/2}\). By the interior Harnack
	inequality and the lower bound for \(p\), we have  
	\begin{equation}\label{eq:w-estimate}
		\sup_{B_3}w
		\leq\sup_{B_3}h
		\leq Ch(0)
		\leq C\bigl(w(0)+F\bigr).
	\end{equation}
	This also covers \(h\equiv0\), in which case \(w\equiv0\) in \(B_{7/2}\).

\medskip

Next, let \(q\) solve
	\begin{equation*}
		\Delta q=w\quad\text{in }B_3,
		\qquad q=0\quad\text{on }\partial B_3,
	\end{equation*}
	and set \(M=\sup_{B_3}w\). Since \(w\ge0\), the maximum principle gives \(q\le0\) in \(B_3\).
To obtain a lower bound, consider 
	\begin{equation*}
		\psi(x)=\frac{M}{6}\bigl(|x|^2-9\bigr).
	\end{equation*}
	Then
\[
    \Delta\psi=M
    \quad\text{in }B_3,
    \qquad
    \psi=0
    \quad\text{on }\partial B_3.
\]
We obtain from the maximum principle that 
	\begin{equation*}
		-\frac32M\leq q\leq0
		\quad\text{in }B_3.
	\end{equation*}
	The function \(H:=U-q\) is positive and harmonic in \(B_3\). A second
	application of Harnack's inequality, followed by \eqref{eq:w-estimate}, gives
	\begin{equation}\label{eq:U-preliminary}
		\sup_{B_2}U
		\leq\sup_{B_2}H
		\leq CH(0)
		\leq C\bigl(U(0)+w(0)+F\bigr).
	\end{equation}

\medskip

It remains to prove that
	\begin{equation}\label{eq:w-origin-bound}
		w(0)\leq C\bigl(U(0)+F\bigr).
	\end{equation}
	Suppose this estimate fails. Then there are admissible pairs \((U_j,F_j)\)
	for which, upon setting \(A_j=\Delta U_j(0)\),
	\begin{equation*}
		\frac{A_j}{U_j(0)+F_j}\longrightarrow\infty.
	\end{equation*}
	In particular, \(A_j>0\) for all sufficiently large \(j\). Define
	\begin{equation*}
		Z_j=\frac{U_j}{A_j},
		\qquad \varepsilon_j=\frac{F_j}{A_j}.
	\end{equation*}
	Then
	\begin{equation}\label{eq:normalized-properties}
		\begin{gathered}
			Z_j>0,\qquad Z_j(0)+\varepsilon_j\longrightarrow0,
			\qquad \Delta Z_j(0)=1,\\
			Z_j\Delta Z_j\geq2|\nabla Z_j|^2,
			\qquad 0\leq\Delta^2Z_j\leq\varepsilon_j
			\quad\text{in }B_4.
		\end{gathered}
	\end{equation}
	Applying \eqref{eq:U-preliminary} to \(U_j\) and dividing by \(A_j\), we obtain
	\begin{equation*}
		\sup_{B_2}Z_j
		\leq C\bigl(Z_j(0)+1+\varepsilon_j\bigr)
		\leq C.
	\end{equation*}
	The interior \(L^4\)-estimate for the biharmonic operator
	(see the local estimate in
	\cite[Theorem~2.20]{GazzolaGrunauSweers2010}) yields
	\begin{equation*}
		\|Z_j\|_{W^{4,4}(B_{3/2})}
		\leq C\left(
			\|Z_j\|_{L^4(B_2)}
			+\|\Delta^2Z_j\|_{L^4(B_2)}
		\right)
		\leq C.
	\end{equation*}
	In dimension three, the embedding
\[
    W^{4,4}(B_{3/2})
    \hookrightarrow C^{2,1/2}(\overline{B_1})
\]
is compact. Hence, after passing to a subsequence, there exists
\(Z\in C^{2,1/2}(\overline{B_1})\) such that
\[
    Z_j\longrightarrow Z
    \quad\text{in }C^2(\overline{B_1}).
\]
	Letting $j\rightarrow \infty$  in \eqref{eq:normalized-properties}, we therefore obtain 
	\begin{equation}\label{eq:limit-properties}
		Z\geq0,\qquad Z(0)=0,\qquad \Delta Z(0)=1,
		\qquad Z\Delta Z\geq2|\nabla Z|^2
		\quad\text{in }B_1.
	\end{equation}

	Because the origin is a minimum point of \(Z\), we have \(\nabla Z(0)=0\),
	and \(D^2Z(0)\) is positive semidefinite. Let
	\(\lambda_1,\lambda_2,\lambda_3\geq0\) be its eigenvalues and
	\(e_1,e_2,e_3\) corresponding unit eigenvectors. Since \(\Delta Z(0)=1\),
	\begin{equation}\label{eq:hessian-trace}
		\lambda_1+\lambda_2+\lambda_3=1.
	\end{equation}
	For each \(i\), the \(C^2\)-regularity of \(Z\) gives, as \(t\to0\),
	\begin{equation*}
		Z(te_i)=\frac{\lambda_i}{2}t^2+o(t^2),
		\qquad
		\nabla Z(te_i)=\lambda_i t e_i+o(|t|),
		\qquad
		\Delta Z(te_i)=1+o(1).
	\end{equation*}
	Substituting these expansions into the last inequality in \eqref{eq:limit-properties}, dividing by \(t^2\), and letting \(t\to0\), we obtain
	\begin{equation*}
		\frac{\lambda_i}{2}\geq2\lambda_i^2.
	\end{equation*}
	Thus \(0\leq\lambda_i\leq1/4\) for each \(i\), so
	\(\lambda_1+\lambda_2+\lambda_3\leq3/4\), contradicting
	\eqref{eq:hessian-trace}. This proves \eqref{eq:w-origin-bound}; combining it
	with \eqref{eq:U-preliminary} gives \eqref{eq:local-curvature-estimate}.
\end{proof}
	
	\medskip
	
	The local estimate above allows us to control rescaled solutions uniformly.
To prepare for the analysis near the puncture and at infinity, we next
establish scale-invariant pointwise bounds and a spherical Harnack inequality.

\begin{lem}\label{lem:curvature-only-estimates}
		Let \(u\in C^4(\mathbb R^3\setminus\{0\})\) be a positive solution of
		\begin{equation*}
			\Delta^2u=u^{-7}
			\quad\text{in }\mathbb R^3\setminus\{0\}.
		\end{equation*}
		Assume that \(g=u^{-4}|dx|^2\) has nonnegative scalar curvature. Then there exist constants \(c,C_0>0\) and \(C_H\geq1\) such that
		\begin{equation}\label{eq:curvature-only-estimates}
			\begin{aligned}
				c|x|^{1/2}&\leq u(x)\leq C_0\max\{1,|x|\} && (x\ne0),\\
				\sup_{\theta\in\mathbb S^2}u(r\theta)
				&\leq C_H\inf_{\theta\in\mathbb S^2}u(r\theta) && (r>0).
			\end{aligned}
		\end{equation}
		%Neither completeness nor any limiting behavior at the origin is assumed.
		\end{lem}
		\begin{proof}
By \eqref{eq:scalar-curvature-identity}, the curvature assumption is
equivalent to
\begin{equation}\label{eq:curvature-only-differential-inequality}
  u\Delta u\geq2|\nabla u|^2.
\end{equation}
Thus \(\Delta u\geq0\), and \(h=u^{-1}\) is superharmonic, since
\[
  \Delta h=u^{-3}\bigl(2|\nabla u|^2-u\Delta u\bigr)\leq0.
\]

\medskip
We divide the proof into four steps.

\medskip
\emph{Step 1: the upper bound.}
Let \(a=\min_{|x|=1}h(x)>0\). 
For \(0<\varepsilon<1\), consider the harmonic function
\[
    \varphi_\varepsilon(x)
    =
    a\,\frac{1-\varepsilon/|x|}{1-\varepsilon}
\]
on \(B_1\setminus\overline{B_\varepsilon}\). Then we have \(h\ge a=\varphi_\varepsilon\) on $|x|=1$, and 
\(
    h>0=\varphi_\varepsilon
\) on $|x|=\varepsilon$.
The minimum
principle and the limit
\(\varepsilon\rightarrow0\) therefore give \(h\geq a\) for
\(0<|x|\leq1\). Similarly, comparison on
\(B_R\setminus\overline{B_1}\) with
\[
  a\frac{|x|^{-1}-R^{-1}}{1-R^{-1}}
\]
is valid because $h\ge a$ on $|x|=1$ and $h>0$ on $|x|=R$.
Letting  \(R\to\infty\), we have  \(h(x)\geq a/|x|\) for
\(|x|\geq1\). Consequently,
\begin{equation}\label{eq:curvature-only-upper-bound}
  u(x)\leq a^{-1}\max\{1,|x|\}.
\end{equation}

\medskip
\emph{Step 2: a Liouville obstruction.}
We claim that there is no positive function \(U\in C^4(\mathbb R^3)\) satisfying
\begin{equation}\label{eq:curvature-only-entire-system}
  \Delta^2U=U^{-7},\qquad
  U\Delta U\geq2|\nabla U|^2
  \quad\text{in }\mathbb R^3.
\end{equation}
Suppose, to the contrary, that such a function exists. Applying the
comparison argument from Step~1 to \(U^{-1}\), with
\(a_0=\min_{|y|=1}U(y)^{-1}>0\), we obtain
\(U(y)\leq a_0^{-1}|y|\) for \(|y|\geq1\).
Since \(U\) is bounded on \(\overline{B_1}\), it follows that
\begin{equation*}
  U(x)\leq C(1+|x|).
\end{equation*}
Set \(W=\Delta U\). Then \(W\geq0\) and \(\Delta W=U^{-7}>0\).
Fix \(x\in\mathbb R^3\). For \(R\geq1\), choose
\(\chi_R\in C_c^\infty(B_{2R}(x))\) such that \(\chi_R=1\) on
\(B_R(x)\), \(0\leq\chi_R\leq1\), and
\(|\Delta\chi_R|\leq CR^{-2}\). The mean-value inequality for the
nonnegative subharmonic function \(W\), followed by two integrations by
parts, gives
\[
  \begin{aligned}
  0\leq W(x)
  &\leq \frac{C}{R^3}\int_{B_R(x)}W
   \leq \frac{C}{R^3}\int_{\mathbb R^3}\chi_R\Delta U\\
  &=\frac{C}{R^3}\int_{\mathbb R^3}U\Delta\chi_R
   \leq\frac{C(1+|x|+R)}{R^2}.
  \end{aligned}
\]
Letting \(R\to\infty\), we obtain \(W(x)=0\). Since \(x\) is arbitrary,
this contradicts \(\Delta W=U^{-7}>0\).

\medskip
\emph{Step 3: the lower bound.}
We first recall a standard doubling-selection argument. Let \(k>0\), and let
\(M>0\) be continuous on \(\mathbb R^3\setminus\{0\}\). If
\(|x_0|M(x_0)>2k\), then there exists \(y\) such that
\begin{equation}\label{eq:curvature-only-doubling}
  |y|M(y)>2k,\qquad
  M(z)\leq2M(y)
  \quad\text{whenever }|z-y|\leq\frac{k}{M(y)}.
\end{equation}
We prove the assertion by an iterative argument. Set \(x_0\) as above,
and suppose that \(x_n\) has been chosen so that
\[
    |x_n|M(x_n)>2k.
\]
In particular,
\[
    \frac{k}{M(x_n)}<\frac{|x_n|}{2},
\]
so the ball \(B_{k/M(x_n)}(x_n)\) does not contain the origin.

If
\[
    M(z)\le2M(x_n)
    \quad\text{for every }
    z\in\overline{B_{k/M(x_n)}(x_n)},
\]
then the iteration terminates, and we take \(y=x_n\). Otherwise, we
may choose \(x_{n+1}\) such that
\[
    |x_{n+1}-x_n|\le\frac{k}{M(x_n)}
    \qquad\text{and}\qquad
    M(x_{n+1})>2M(x_n).
\]
Since
\[
    |x_{n+1}|
    \ge |x_n|-|x_{n+1}-x_n|
    \ge |x_n|-\frac{k}{M(x_n)},
\]
we have 
\[
    \begin{aligned}
        |x_{n+1}|M(x_{n+1})
        &>
        2M(x_n)
        \left(|x_n|-\frac{k}{M(x_n)}\right)\\
        &=
        2\bigl(|x_n|M(x_n)-k\bigr)
        >2k.
    \end{aligned}
\]
Thus the induction hypothesis is preserved at every step.

Suppose that the iteration never terminates. Then
\[
    M(x_n)\geq2^nM(x_0)\qquad(n\geq0),
\]
and hence
\[
    |x_{n+1}-x_n|
    \le\frac{k}{M(x_n)}
    \leq\frac{k}{2^nM(x_0)}.
\]
It follows that
\[
    \sum_{n=0}^{\infty}|x_{n+1}-x_n|
    \le\frac{2k}{M(x_0)}
    <|x_0|.
\]
Therefore, \((x_n)\) is a Cauchy sequence and converges to some
\(x_\infty\in\mathbb R^3\). Moreover,
\[
    |x_\infty|
    \ge
    |x_0|-\sum_{n=0}^{\infty}|x_{n+1}-x_n|
    \ge
    |x_0|-\frac{2k}{M(x_0)}
    >0.
\]
Hence \(x_\infty\in\mathbb R^3\setminus\{0\}\). By continuity of \(M\),
\[
    M(x_n)\longrightarrow M(x_\infty)<\infty,
\]
which contradicts \(M(x_n)\geq2^nM(x_0)\to\infty\). Thus the iteration
must terminate after finitely many steps, and the terminal point
satisfies \eqref{eq:curvature-only-doubling}.

Suppose now that the desired lower bound fails and set \(M=u^{-2}\). Then
\(\sup_{x\ne0}|x|M(x)=\infty\). Applying
\eqref{eq:curvature-only-doubling} with \(k=j\), choose \(y_j\) such that
\[
    |y_j|M(y_j)>2j
\]
and
\[
    M(z)\le2M(y_j)
    \quad\text{whenever }
    |z-y_j|\le\frac{j}{M(y_j)}.
\]
Set \(\lambda_j=M(y_j)^{-1}=u(y_j)^2\).
Since \(|y_j|>2j\lambda_j\), the ball \(B_{j\lambda_j}(y_j)\)
does not contain the origin. Define
\[
  U_j(z)=\frac{u(y_j+\lambda_jz)}{u(y_j)},
  \qquad z\in B_j.
\]
The doubling property and the scaling invariance of the equation and
the curvature inequality imply
\begin{equation}\label{eq:curvature-only-rescaling}
  U_j(0)=1,\qquad U_j\geq2^{-1/2},\qquad
  \Delta^2U_j=U_j^{-7},\qquad
  U_j\Delta U_j\geq2|\nabla U_j|^2
  \quad\text{in }B_j.
\end{equation}
Fix \(R>0\) and take \(j>4R\). Applying
Lemma~\ref{lem:local-curvature-estimate} to
\(V(\xi)=U_j(R\xi)\), \(\xi\in B_4\), we obtain
\[
  \sup_{B_R}U_j\leq C_R,
\]
since  \(\Delta V(0)\geq0\) and
\(0<\Delta^2V\leq2^{7/2}R^4\). 
Together with the lower bound in \eqref{eq:curvature-only-rescaling},
these upper bounds allow us to apply interior \(W^{4,p}\)-estimates.
Fixing \(p>3\) and using Sobolev--Morrey embedding, we obtain uniform
\(C^{3,\alpha}\)-bounds on smaller compact sets, where \(\alpha=1-3/p\).
Since \(U_j\geq2^{-1/2}\), the nonlinear terms \(U_j^{-7}\) are uniformly
bounded in \(C^{3,\alpha}\) there. Interior Schauder estimates then yield uniform
\(C^{4,\alpha}\)-bounds on every fixed compact set. A diagonal
argument therefore yields, after passing to a subsequence,
\[
  U_j\longrightarrow U\qquad\text{in }C^4_{\mathrm{loc}}(\mathbb R^3).
\]
The limit \(U\)  is a positive entire solution of
\eqref{eq:curvature-only-entire-system}, contradicting the Liouville
obstruction. Hence
\begin{equation}\label{eq:curvature-only-lower-bound}
  u(x)\geq c|x|^{1/2}\qquad(x\ne0).
\end{equation}

\medskip
\emph{Step 4: the spherical Harnack inequality.}
For \(\rho>0\), set
\[
  u_\rho(y):=\rho^{-1/2}u(\rho y).
\]
This rescaling preserves the equation and
\eqref{eq:curvature-only-differential-inequality}. By
\eqref{eq:curvature-only-lower-bound}, \(u_\rho\geq c/\sqrt2\) on
\(1/2\leq|y|\leq3/2\). Choose \(\delta>0\) such that
\[
    B_{4\delta}(z)
    \subset
    \left\{y:\frac12<|y|<\frac32\right\}
    \qquad\text{for every }z\in\mathbb S^2.
\]
Fix \(z\in\mathbb S^2\) and set
\[
    V(\xi):=u_\rho(z+\delta\xi),
    \qquad \xi\in B_4.
\]
Applying Lemma~\ref{lem:local-curvature-estimate} to \(V\), we obtain
\begin{equation}\label{eq:curvature-only-local-harnack}
  \sup_{B_\delta(z)}u_\rho
  \leq C\left[u_\rho(z)+\delta^4
  \left(\frac{\sqrt2}{c}\right)^7\right]
  \leq C_*u_\rho(z),
\end{equation}
where the last inequality uses \(u_\rho(z)\geq c\), and \(C_*\) is
independent of \(\rho\) and \(z\).

Let \(z_0,z_N\in\mathbb S^2\) be points at which \(u_\rho\) attains
its minimum and maximum, respectively. Join \(z_0\) to \(z_N\) by a
minimizing great-circle arc. Since its length is at most \(\pi\), it
can be subdivided into a fixed number \(N=N(\delta)\) of subarcs with
endpoints
\(
    z_0,z_1,\ldots,z_N
\)
satisfying
\(
    |z_{k+1}-z_k|<\delta.
\)
Applying \eqref{eq:curvature-only-local-harnack}, we  get 
\[
    u_\rho(z_{k+1})
    \le C_*u_\rho(z_k),
    \qquad 0\le k<N.
\]
Hence
\[
    \sup_{\mathbb S^2}u_\rho
    =
    u_\rho(z_N)
    \le
    C_*^Nu_\rho(z_0)
    =
    C_*^N\inf_{\mathbb S^2}u_\rho.
\]
Since \(N\) is independent of \(\rho\), rescaling back yields
\[
    \sup_{\theta\in\mathbb S^2}u(r\theta)
    \le
    C_H\inf_{\theta\in\mathbb S^2}u(r\theta),
    \qquad
    C_H:=C_*^N.
\]
 Combining this estimate with
\eqref{eq:curvature-only-upper-bound} and
\eqref{eq:curvature-only-lower-bound}, we prove
\eqref{eq:curvature-only-estimates}.
\end{proof}
	
	\medskip
	
	The pointwise bounds and the spherical Harnack inequality provide control
of the size and angular oscillation of the solution on each sphere.
In cylindrical variables, this becomes uniform positivity of the profile
and comparability with its spherical mean, which form the starting point
for the radial symmetry argument below.

\begin{lem}\label{lem:radial-symmetry-vanishing-origin}
		Let \(u\in C^4(\mathbb R^3\setminus\{0\})\) be a positive solution of
		\eqref{eq1}. If the conformal metric \(g=u^{-4}|dx|^2\) has
		nonnegative scalar curvature, then \(u\) is radially symmetric about the
		origin.
		\end{lem}
		\begin{proof}
			We divide the proof into three steps.

			\emph{Step 1: cylindrical bounds and the behavior at the negative end.} Set $\mathcal C=\mathbb R\times\mathbb S^2$ and define
		\begin{equation*}
			V(t,\theta)=e^{-t/2}u(e^t\theta),\qquad
			m(t)=\frac1{4\pi}\int_{\mathbb S^2}V(t,\theta)\,d\sigma.
		\end{equation*}
			The Emden--Fowler transformation converts dilations in
\(\mathbb R^3\setminus\{0\}\) into translations on \(\mathcal C\) and
transforms the equation into
			\begin{equation}\label{eq:rs-zero-cylinder}
				PV=V^{-7},\qquad 
				P=\left(\partial_t^2+\Delta_{\mathbb S^2}+\frac34\right)^2
				-4\partial_t^2.
			\end{equation}
		 By Lemma~\ref{lem:curvature-only-estimates}, \(V\ge c>0\) on
\(\mathcal C\). Moreover, the spherical Harnack inequality and the
definition of \(m\) imply
			\begin{equation*}
				C_H^{-1}m(t)\leq V(t,\theta)\leq C_Hm(t).
			\end{equation*}
			The upper bound for \(u\) implies \(V(t,\theta)\leq C_0e^{|t|/2}\). Thus
			\begin{equation}\label{eq:rs-zero-mean-comparison}
				V\geq c,\qquad
				C_H^{-1}m(t)\leq V(t,\theta)\leq C_Hm(t),\qquad
			m(t)\leq C_0e^{|t|/2}.
		\end{equation}
		The spherical mean isolates the possible exponential modes at the two
		ends. Averaging \eqref{eq:rs-zero-cylinder} over \(\mathbb S^2\), we obtain
		\begin{equation*}
			\begin{aligned}
				m^{(4)}-\frac52m''+\frac9{16}m&=F(t),\\
				F(t):=\frac1{4\pi}\int_{\mathbb S^2}V(t,\theta)^{-7}\,d\sigma,
				&\qquad 0<F(t)\leq c^{-7}.
			\end{aligned}
		\end{equation*}
			Write $D=d/dt$ and
			\begin{equation*}
				\mathcal L=D^4-\frac52D^2+\frac9{16}
				=\left(D^2-\frac14\right)\left(D^2-\frac94\right).
			\end{equation*}
			The function
			\begin{equation*}
				G(t)=\frac12e^{-|t|/2}-\frac16e^{-3|t|/2},
			\end{equation*}
			satisfies \(\mathcal LG=\delta_0\) in \(\mathcal D'(\mathbb R)\). Since
			\(G,G',G''\in L^1(\mathbb R)\) and \(F\in L^\infty(\mathbb R)\), the
			convolution
			\begin{equation*}
				p(t)=(G*F)(t)
			\end{equation*}
			is well defined. Moreover, \(p^{(k)}=G^{(k)}*F\) for \(0\leq k\leq2\), \(\mathcal Lp=F\) in \(\mathcal D'(\mathbb R)\), and
			\begin{equation*}
				\sum_{k=0}^2\|p^{(k)}\|_{L^\infty(\mathbb R)}
				\leq C\|F\|_{L^\infty(\mathbb R)}.
			\end{equation*}
			Therefore \(\mathcal L(m-p)=0\) in \(\mathcal D'(\mathbb R)\).
Distributional solutions of this constant-coefficient ordinary differential
equation are smooth, and the kernel is spanned by \(e^{3t/2},e^{t/2},e^{-t/2},e^{-3t/2}\). Hence there exist constants \(A_\pm,E_\pm\) such that
			\begin{equation*}
				m(t)=p(t)+A_+e^{3t/2}+E_+e^{t/2}
				+E_-e^{-t/2}+A_-e^{-3t/2}.
			\end{equation*}
			The growth bound in \eqref{eq:rs-zero-mean-comparison}, together
with the boundedness of \(p\), forces
			\(A_+=A_-=0\). Since \(m>0\), we must also have \(E_+,E_-\geq0\); otherwise the corresponding exponential term would force \(m\) to
become negative at one of the two ends.
 Consequently, we obtain 
			\begin{equation}\label{eq:rs-zero-mean-expansions}
				\begin{aligned}
					 m(t)&=E_+e^{t/2}+O(1)
        &&\text{as }t\to+\infty,\\
        m(t)&=E_-e^{-t/2}+O(1)
        &&\text{as }t\to-\infty.
			\end{aligned}
		\end{equation}
			The rupture condition removes the growing mode at the negative end:
		\begin{equation}\label{eq:rs-zero-negative-coefficient}
			e^{t/2}m(t)
			=\frac1{4\pi}\int_{\mathbb S^2}u(e^t\theta)\,d\sigma
			\longrightarrow0
			\qquad (t\to-\infty).
		\end{equation}
			Combining \eqref{eq:rs-zero-mean-expansions} with
			\eqref{eq:rs-zero-negative-coefficient}, we get  \(E_-=0\). Hence \(m\),
			and therefore \(V\), is bounded on
			\((-\infty,0]\times\mathbb S^2\). Equivalently, there are
			\(r_0,C>0\) such that
			\[
				u(x)\leq C|x|^{1/2}\qquad(0<|x|<r_0).
			\]
			
\medskip
			\emph{Step 2: control of the positive end.}
			If \(E_+=0\), then \eqref{eq:rs-zero-mean-comparison} and
			\eqref{eq:rs-zero-mean-expansions} show that \(V\) is bounded on the
			whole cylinder. Suppose now that \(E_+=E>0\). Since
			\begin{equation*}
				m(t)=E e^{t/2}+O(1)
				\quad\text{as }t\to+\infty,
			\end{equation*}
			the comparison between \(V\) and \(m\) yields constants \(c_1,C_1,r_0>0\) such that
		\begin{equation}\label{eq:rs-zero-linear-comparison}
			c_1|x|\leq u(x)\leq C_1|x|
			\qquad (|x|\geq r_0).
		\end{equation}
			We shall prove that the sharper expansion
		\begin{equation}\label{eq:rs-zero-linear-expansion}
			u(x)=E|x|+O(1)
			\qquad (|x|\to\infty)
		\end{equation}
			holds uniformly in direction.

			 Define
			\(I(y)=y/|y|^2\) and the Kelvin transform
			\(\widehat u(y)=|y|u(I(y))\). By
			\eqref{eq:rs-zero-linear-comparison}, after decreasing
			\(\rho>0\) if necessary, we have 
		\begin{equation*}
			0<c_1\leq\widehat u(y)\leq C_1
			\qquad (0<|y|<\rho).
		\end{equation*}
			The biharmonic Kelvin transformation and the pullback identity give
		\begin{equation*}
			\Delta^2\widehat u=\widehat u^{-7},
			\qquad
			\widehat g:=\widehat u^{-4}|dy|^2=I^*g.
		\end{equation*}
			Thus \(R_{\widehat g}\geq0\), and the scalar-curvature identity implies
		\begin{equation*}
			q:=\Delta\widehat u\geq0
			\quad\text{in }B_\rho\setminus\{0\}.
		\end{equation*}

			Extend \(f:=\widehat u^{-7}\) arbitrarily to the origin. Then
			\(f\in L^\infty(B_\rho)\). Fix \(p_0>3\), and let
			\(\Psi\in W^{2,p_0}(B_\rho)\) solve
		\begin{equation*}
			\Delta\Psi=f\quad\text{in }B_\rho,
			\qquad
			\Psi=0\quad\text{on }\partial B_\rho.
		\end{equation*}
			By Sobolev embedding, \(\Psi\in
			C^{1,\alpha}(\overline{B_\rho})\) for some \(0<\alpha<1\).
			On the punctured ball, there holds
		\begin{equation*}
			\Delta(q-\Psi)=\Delta^2\widehat u-f=0.
		\end{equation*}
			Moreover, \(q-\Psi\) is bounded below because \(q\geq0\) and
			\(\Psi\) is bounded. Choose \(K>0\) so that \(q-\Psi+K>0\).
			B\^ocher's theorem
			\cite[Theorem~3.9]{AxlerBourdonRamey2001} implies that
		\[
			q-\Psi+K=\frac{a}{|y|}+H,
			\qquad a\geq0,
		\]
			where \(H\) extends harmonically across the origin. Consequently,
		\begin{equation*}
			q(y)=\frac{a}{|y|}+Q(y),
			\qquad a\geq0,
		\end{equation*}
			where \(Q=\Psi+H-K\) extends to a \(C^{1,\alpha}\) function near
			the origin. After decreasing \(\rho\) if necessary, we may assume
			that \(Q\in C^{1,\alpha}(\overline{B_\rho})\).

			Since \(\Delta|y|=2/|y|\) in
			\(\mathbb R^3\setminus\{0\}\), define 
		\begin{equation*}
			v(y):=\widehat u(y)-\frac{a}{2}|y|.
		\end{equation*}
			Then \(v\) is bounded on the punctured ball and
			\(\Delta v=Q\). Let
			\(z\in W^{2,p_0}(B_\rho)\cap W_0^{1,p_0}(B_\rho)\) solve
			\(\Delta z=Q\) in \(B_\rho\), with \(z=0\) on
			\(\partial B_\rho\). The function \(v-z\) is bounded
			and harmonic on \(B_\rho\setminus\{0\}\), so its singularity at
			the origin is removable by
			\cite[Theorem~2.3]{AxlerBourdonRamey2001}.
			Consequently, after decreasing \(\rho>0\)
			if necessary, \(v\) can be extended to a function in
			\(C^{1,\alpha}(B_\rho)\), and
		\begin{equation}\label{eq:rs-zero-kelvin-decomposition}
			\widehat u(y)=v(y)+\frac a2|y|
			\qquad (0<|y|<\rho).
		\end{equation}
			In particular, for some constant \(L\), we have
		\begin{equation}\label{eq:rs-zero-kelvin-first-order}
			\widehat u(y)=L+O(|y|)
			\qquad (y\to0),
		\end{equation}
			uniformly in direction.

			It remains to determine \(L\). By the definitions of
			\(\widehat u\), \(V\), and \(m\), we have 
		\begin{equation*}
			\frac1{4\pi}\int_{\mathbb S^2}
			\widehat u(s\theta)\,d\sigma
			=s^{1/2}m(-\log s)\longrightarrow E
			\qquad (s\rightarrow0).
		\end{equation*}
			The expansion \eqref{eq:rs-zero-kelvin-first-order} shows that
			the left-hand side tends to \(L\), and hence \(L=E\). Therefore
			\(\widehat u(y)=E+O(|y|)\), uniformly as \(y\to0\). Since
			\(|I(x)|=|x|^{-1}\) and
			\(u(x)=|x|\widehat u(I(x))\), this proves
			\eqref{eq:rs-zero-linear-expansion}. Consequently,
		\begin{equation*}
			V(t,\theta)=E e^{t/2}+O(e^{-t/2})
			\qquad (t\to+\infty),
		\end{equation*}
			uniformly in \(\theta\).

\medskip
\emph{Step 3: elimination of the angular dependence.}
We now prove that every rotational derivative of \(V\) vanishes.
We first make this derivative precise. Every one-parameter subgroup
of \(\mathrm{SO}(3)\) can be written as
\[
  \mathcal R_s=e^{sB},
  \qquad B^{\mathsf T}=-B,
\]
for a skew-symmetric matrix \(B\). Its infinitesimal generator on
\(\mathbb S^2\) is the vector field
\begin{equation*}
  X_B(\theta)
  :=\left.\frac{d}{ds}\right|_{s=0}\mathcal R_s\theta
  =B\theta,
\end{equation*}
and  $X_B$ is tangent to the sphere, because skew-symmetry yields
\[
  \theta\cdot X_B(\theta)=\theta\cdot B\theta=0.
\]
Moreover, the flow of \(X_B\) is precisely the family \(\mathcal R_s\), and
each \(\mathcal R_s\) is an isometry of \(\mathbb S^2\). Thus \(X_B\) is a
Killing vector field. For a smooth function \(f\) on \(\mathbb S^2\), its
derivative in this rotational direction is
\begin{equation*}
  X_Bf(\theta)
  :=\left.\frac{d}{ds}\right|_{s=0}f(\mathcal R_s\theta)
  =\nabla_{\mathbb S^2}f(\theta)\cdot B\theta.
\end{equation*}
Fix such a matrix \(B\), write \(X=X_B\), and let \(X\) act only on the
\(\theta\)-variable of functions on \(\mathcal C\). Define
\[
  \phi(t,\theta)
  =\left.\frac{d}{ds}\right|_{s=0}V(t,\mathcal R_s\theta)
  =XV(t,\theta).
\]
For every compact subset $K\Subset\mathcal C$, the positivity and 
continuity of $V$ imply that
\(
\inf_K V>0.
\)
Thus the singularity of the map $s\mapsto s^{-7}$ at $s=0$ does not arise 
locally. Standard interior elliptic regularity and
bootstrapping applied to \eqref{eq:rs-zero-cylinder} yield
\(
V\in C^\infty(\mathcal C).
\)
 Since rotations preserve the round metric, we have
\(X\Delta_{\mathbb S^2}f=\Delta_{\mathbb S^2}(Xf)\) and 
\(X\partial_t f=\partial_t(Xf)\). Hence \(X\) commutes with the 
cylindrical operator \(P\). Consequently, applying \(X\) to \eqref{eq:rs-zero-cylinder} and using
\(X(V^{-7})=-7V^{-8}XV\), we find that $\phi=XV$ satisfies
\begin{equation}\label{eq:linearized-equation}
  P\phi+7V^{-8}\phi=0
  \qquad\text{on }\mathcal C.
\end{equation}

We first establish the uniform bounds needed in the cutoff argument. For
\(\tau\in\RR\), set
\[
  Q_\tau=(\tau-2,\tau+2)\times\mathbb S^2,
  \qquad
  Q_\tau'=(\tau-1,\tau+1)\times\mathbb S^2.
\]
Fix \(p>3\). Since the product metric and the coefficients of \(P\) are
invariant under translations in \(t\), the interior estimate
\begin{equation}\label{eq:local-estimate-V}
  \|V\|_{W^{4,p}(Q_\tau')}
  \leq C_p\left(
    \|V^{-7}\|_{L^p(Q_\tau)}+\|V\|_{L^p(Q_\tau)}
  \right)
\end{equation}
holds with a constant \(C_p\) independent of \(\tau\). By Step~1 and
\eqref{eq:rs-zero-mean-comparison}, for all sufficiently negative
\(\tau\), we have
\[
    c\leq V\leq C
    \quad\text{on }Q_\tau.
\]
Consequently, by \eqref{eq:local-estimate-V} and the embedding
\(W^{4,p}(Q_\tau')\hookrightarrow C^{3,\beta}(Q_\tau')\), where
\(\beta=1-3/p>0\), 
the angular derivatives of \(V\) are uniformly bounded
as \(t\to-\infty\). 
Hence \(\phi=XV\) is bounded at the negative end. If
\(E_+=0\), Step~2 shows that \(V\) is bounded on the whole cylinder, and the
same argument bounds \(\phi\) at the positive end.

Suppose now that \(E_+>0\), and set \(r=e^{-t}\). For all sufficiently large
\(t\), it follows from \eqref{eq:rs-zero-kelvin-decomposition} that
\[
  V(t,\theta)
  =r^{-1/2}\widehat u(r\theta)
  =r^{-1/2}v(r\theta)+\frac a2r^{1/2}.
\]
The last term is independent of \(\theta\). Differentiating the first term
along the rotation \(\mathcal R_s\), we obtain
\[
  \phi(t,\theta)
  =r^{1/2}\nabla v(r\theta)\cdot B\theta.
\]
Since \(v\in C^{1,\alpha}(B_\rho)\) and \(B\theta\) is bounded on
\(\mathbb S^2\), we have 
\[
  |\phi(t,\theta)|\leq C_Br^{1/2}=C_Be^{-t/2}
  \qquad\text{as }t\to+\infty, 
\]
uniformly in \(\theta\). Combining the estimates at both ends with the
smoothness of \(\phi\) on compact subsets of \(\mathcal C\), we conclude that
\begin{equation}\label{eq:angular-bounded}
  \|\phi\|_{L^\infty(\mathcal C)}<\infty.
\end{equation}

Set \(q:=7V^{-8}\). By \eqref{eq:rs-zero-mean-comparison},
\(0<q\leq7c^{-8}\). Applying the interior estimate to
\(P\phi=-q\phi\) and using \eqref{eq:angular-bounded}, we obtain
\[
  \|\phi\|_{W^{4,p}(Q_\tau')}
  \leq C_p\left(
    \|q\phi\|_{L^p(Q_\tau)}+\|\phi\|_{L^p(Q_\tau)}
  \right)\leq C_B,
\]
uniformly in
\(\tau\).
We deduce from the Sobolev embedding that 
\begin{equation}\label{eq:uniform-derivative-bound}
  \sup_{\mathcal C}
  \bigl(
    |\phi|+|\phi_t|+|\phi_{tt}|+|\nabla_{\mathbb S^2}\phi|
  \bigr)<\infty.
\end{equation}

\medskip
With these bounds in hand, we derive a cutoff energy identity.
Set
\[
  \mathcal A=\Delta_{\mathbb S^2}+\frac34,
  \qquad
  P=\partial_t^4+2\mathcal A\partial_t^2+\mathcal A^2-4\partial_t^2.
\]
Choose \(\eta\in C_c^\infty(\RR)\) such that
\[
  0\leq\eta\leq1,
  \qquad
  \eta=1\ \text{on }[-1,1],
  \qquad
  \operatorname{supp}\eta\subset[-2,2],
\]
and define \(\eta_R(t)=\eta(t/R)\) for \(R\geq1\). Because \(\eta_R\) is
compactly supported, all integrations by parts below have no boundary terms at
\(t=\pm\infty\). We write \(\int\) for integration over \(\mathcal C\) with
respect to \(dt\,d\sigma\).

We compute separately the four contributions arising from \(P\).
Using the compact support of \(\eta_R\), the self-adjointness of
\(\mathcal A\) on \(\mathbb S^2\), and repeated integration by parts
in \(t\), we obtain
\begin{align}
  \int\eta_R\phi\,\phi_{tttt}
  &={}
  \int\eta_R\phi_{tt}^2
  +\int\eta_R''\bigl(\phi\phi_{tt}-\phi_t^2\bigr),
  \label{eq:ibp-one}
\end{align}
and
\begin{align}
  2\int\eta_R\phi\,\mathcal A\phi_{tt}
  &={}
  2\int\eta_R|\nabla_{\mathbb S^2}\phi_t|^2
  -\frac32\int\eta_R\phi_t^2
  +\int\eta_R''
    \left(-|\nabla_{\mathbb S^2}\phi|^2+\frac34\phi^2\right),
  \label{eq:ibp-two}
\end{align}
The remaining two terms satisfy
\begin{align}
  -4\int\eta_R\phi\phi_{tt}
  &={}
  4\int\eta_R\phi_t^2-2\int\eta_R''\phi^2,
  \label{eq:ibp-three}\\
  \int\eta_R\phi\,\mathcal A^2\phi
  &={}
  \int\eta_R|\mathcal A\phi|^2.
  \label{eq:ibp-four}
\end{align}
We have also used the identity
\[
    \int_{\mathbb S^2}
    f\Delta_{\mathbb S^2}g\,d\sigma
    =
    -\int_{\mathbb S^2}
    \nabla_{\mathbb S^2}f\cdot
    \nabla_{\mathbb S^2}g\,d\sigma.
\]

Multiplying \eqref{eq:linearized-equation} by \(\eta_R\phi\), integrating,
and using \eqref{eq:ibp-one}--\eqref{eq:ibp-four}, we obtain the exact identity
\begin{equation}\label{eq:coercive-identity}
  \begin{aligned}
  \mathcal E_R+7\int\eta_RV^{-8}\phi^2
  =-\int\eta_R''\biggl(
    \phi\phi_{tt}-\phi_t^2-\frac54\phi^2
    -|\nabla_{\mathbb S^2}\phi|^2
  \biggr),
  \end{aligned}
\end{equation}
where
\begin{equation}\label{eq:energy}
  \mathcal E_R=
  \int\eta_R\left(
    \phi_{tt}^2+2|\nabla_{\mathbb S^2}\phi_t|^2
    +\frac52\phi_t^2+|\mathcal A\phi|^2
  \right).
\end{equation}
Every term on the left-hand side of \eqref{eq:coercive-identity} is
nonnegative. Taking the absolute value of the right-hand side and using
\eqref{eq:uniform-derivative-bound}, we obtain
\[
  \begin{aligned}
  0\leq \mathcal E_R
  &\leq \mathcal E_R+7\int\eta_RV^{-8}\phi^2\\
  &\leq C_B\int_{\RR}|\eta_R''(t)|\,dt.
  \end{aligned}
\]
Moreover,
\[
  \|\eta_R''\|_{L^1(\RR)}
  =R^{-1}\|\eta''\|_{L^1(\RR)}.
\]
Consequently,
\begin{equation}\label{eq:energy-decay}
  0\leq \mathcal E_R\leq\frac{C_B}{R},
\end{equation}
where \(C_B\) is independent of \(R\).

\medskip
We now combine the spectral gap of \(\mathcal A\) with the energy estimate
to prove that
$
    \phi\equiv0
  $ \text{on } $\mathcal C=\mathbb R\times\mathbb S^2$.
Let \(\{Y_{\ell,k}\}\) be an orthonormal basis of spherical harmonics, with
\[
  -\Delta_{\mathbb S^2}Y_{\ell,k}
  =\ell(\ell+1)Y_{\ell,k}.
\]
If \(\psi=\sum_{\ell,k}c_{\ell,k}Y_{\ell,k}\), then
\[
  \|\mathcal A\psi\|_{L^2(\mathbb S^2)}^2
  =\sum_{\ell,k}
   \left(\frac34-\ell(\ell+1)\right)^2|c_{\ell,k}|^2.
\]
Since
\(\left|3/4-\ell(\ell+1)\right|\geq3/4\) for every integer
\(\ell\geq0\), using Parseval's identity, we have 
\begin{equation}\label{eq:spectral-gap}
  \int_{\mathbb S^2}|\mathcal A\psi|^2\,d\sigma
  \geq\frac9{16}
  \int_{\mathbb S^2}|\psi|^2\,d\sigma
\end{equation}
for every smooth \(\psi\) on \(\mathbb S^2\).

Fix \(L>0\). For every \(R>L\), one has \(\eta_R=1\) on \([-L,L]\).
Applying \eqref{eq:spectral-gap} to \(\phi(t,\cdot)\) for each
\(t\in[-L,L]\), and then using \eqref{eq:energy} and
\eqref{eq:energy-decay}, we find
\[
  0\leq
  \frac9{16}
  \int_{[-L,L]\times\mathbb S^2}\phi^2\,dt\,d\sigma
  \leq \mathcal E_R
  \leq\frac{C_B}{R}.
\]
Letting \(R\to\infty\), we obtain
\[
  \int_{[-L,L]\times\mathbb S^2}\phi^2\,dt\,d\sigma=0.
\]
Since \(\phi\) is smooth, it follows that \(\phi\equiv0\) on
\([-L,L]\times\mathbb S^2\). As \(L>0\) is arbitrary, we deduce that 
\[
  X_BV\equiv0\qquad\text{on }\mathcal C.
\]

Finally, the skew-symmetric matrix \(B\) was arbitrary. Hence
\(XV=0\) for every rotational vector field \(X\) on
\(\mathbb S^2\). Such vector fields can be written as
\[
    X_a(\theta)=a\times\theta,
    \qquad a\in\mathbb R^3.
\]
If
\(\xi\in T_\theta\mathbb S^2\), choosing
\(
    a=\theta\times\xi
\)
and using \(\theta\cdot\xi=0\), we obtain
\[
    X_a(\theta)
    =
    (\theta\times\xi)\times\theta
    =
    \xi.
\]
Thus the rotational directions span \(T_\theta\mathbb S^2\) at
every \(\theta\in\mathbb S^2\), and therefore 
\[
    \nabla_{\mathbb S^2}V\equiv0
    \quad\text{on }\mathcal C,
\]
so \(V\) is independent of \(\theta\). Recalling that
\[
    u(r\theta)=r^{1/2}V(\log r,\theta),
\]
we conclude that \(u\) is radially symmetric about the origin.

		\end{proof}
	
	\medskip
	
	Radial symmetry removes the angular dependence from the cylindrical
equation, leaving an autonomous fourth-order ODE. The bounds already
established at the negative end can now be refined to determine the precise
asymptotic profile at the puncture.

\begin{lem}\label{lem:radial-origin-profile}
		Let \(u\in C^4(\mathbb R^3\setminus\{0\})\) be a positive solution of
		\eqref{eq1}. Assume that \(g=u^{-4}|dx|^2\) has nonnegative scalar
		curvature. By
		Lemma~\ref{lem:radial-symmetry-vanishing-origin}, write \(u=u(r)\). Then
		\begin{equation}\label{eq:radial-origin-profile}
			\lim_{r\downarrow0}\frac{u(r)}{\sqrt r}=c_*,
			\qquad c_*=\left(\frac43\right)^{1/4}.
		\end{equation}
	\end{lem}
	\begin{proof}
	Set
	\[
		v(t)=e^{-t/2}u(e^t),\qquad t\in\mathbb R.
	\]
	Then the cylindrical equation \eqref{eq:rs-zero-cylinder} reduces to
	\begin{equation}\label{eq:radial-origin-profile-ode}
		\mathcal Lv:=v^{(4)}-\frac52v''+\frac9{16}v=v^{-7}
		\qquad\text{on }\mathbb R.
	\end{equation}
	Lemma~\ref{lem:curvature-only-estimates} provides a uniform positive
	lower bound for \(v\), while Step~1 of the proof of
	Lemma~\ref{lem:radial-symmetry-vanishing-origin} establishes its boundedness
	on \((-\infty,0]\). Thus there exist constants \(c,C>0\) such that
	\begin{equation}\label{eq:radial-origin-profile-bounds}
		0<c\leq v(t)\leq C\qquad(t\leq0).
	\end{equation}

	We first show that every entire solution of
	\eqref{eq:radial-origin-profile-ode} that is bounded above and uniformly
	away from zero must equal \(c_*\). Suppose that
	\(z\in C^4(\mathbb R)\) solves
	\eqref{eq:radial-origin-profile-ode} and satisfies
	\begin{equation*}
		0<c_0\leq z(t)\leq C_0\qquad(t\in\mathbb R).
	\end{equation*}
	Since \(z\) and \(z^{-7}\) are uniformly bounded, interior estimates
	for \(\mathcal L\), applied on intervals of fixed length, imply that
	\(z'\) and \(z''\) are uniformly bounded on \(\mathbb R\).

	Fix \(\tau\in\mathbb R\) and set \(w(t)=z(t+\tau)-z(t)\).
	Subtracting the equations for \(z(t+\tau)\) and \(z(t)\), and applying
	the fundamental theorem of calculus to the nonlinearity, we obtain
	\begin{equation}\label{eq:radial-origin-translation-equation}
		w^{(4)}-\frac52w''
		+\left(\frac9{16}+a_\tau(t)\right)w=0,
	\end{equation}
	where
	\[
		a_\tau(t)
		=7\int_0^1
		\bigl((1-s)z(t)+sz(t+\tau)\bigr)^{-8}\,ds>0.
	\]
	Choose \(0\leq\eta\leq1\) in \(C_c^\infty(\mathbb R)\), with
	\(\eta=1\) on \([-1,1]\) and
	\(\operatorname{supp}\eta\subset[-2,2]\), and set
	\(\eta_R(t)=\eta(t/R)\). Multiplying
	\eqref{eq:radial-origin-translation-equation} by \(\eta_Rw\) and
	integrating by parts, we obtain
	\begin{equation*}
		\begin{aligned}
			&\int_{\mathbb R}\eta_R
			\left[(w'')^2+\frac52(w')^2+
			\left(\frac9{16}+a_\tau\right)w^2\right]dt\\
			&\qquad
			=-\int_{\mathbb R}\eta_R''
			\left(ww''-(w')^2-\frac54w^2\right)dt.
		\end{aligned}
	\end{equation*}
	The uniform bounds for \(z,z',z''\) also bound \(w,w',w''\).
	Consequently, the absolute value of the right-hand side is at most
	\(C_\tau R^{-1}\), since
	\[
		\|\eta_R''\|_{L^1(\mathbb R)}
		=R^{-1}\|\eta''\|_{L^1(\mathbb R)}.
	\]
	Because \(\eta_R=1\) on \([-R,R]\) and all terms on the left-hand side
	are nonnegative, for every fixed \(L>0\) and \(R>L\) we have
	\[
		\frac9{16}\int_{-L}^{L}w^2\,dt\leq\frac{C_\tau}{R}.
	\]
	Letting \(R\to\infty\), we conclude that \(w\equiv0\) by continuity.
Thus \(z(t+\tau)=z(t)\) for all \(t\in\mathbb R\).
Since \(\tau\in\mathbb R\) was arbitrary, \(z\) is constant.
Substituting this constant into \eqref{eq:radial-origin-profile-ode},
we obtain \(\frac9{16}z=z^{-7}\). Since \(z>0\), it follows that
\[
    z
    =\left(\frac43\right)^{1/4}=c_*.
\]

	We now determine the limit of \(v\) as \(t\to-\infty\).
	Let \(t_j\to-\infty\) be arbitrary and set \(v_j(s)=v(t_j+s)\).
	For each fixed \(R>0\), we have \(t_j+2R\leq0\) for all sufficiently
	large \(j\), and hence \eqref{eq:radial-origin-profile-bounds} implies
	\[
		c\leq v_j(s)\leq C\qquad(s\in[-2R,2R]).
	\]
	Each \(v_j\) satisfies the same autonomous equation.
	Interior \(W^{4,2}\)-estimates and the one-dimensional Sobolev embedding
	therefore yield uniform \(C^{3,1/2}\)-bounds on smaller intervals.
	Using the equation and the positive lower bound, we then obtain uniform
	\(C^{4,1/2}\)-bounds on compact intervals.

	By the Arzel\`a--Ascoli theorem and a diagonal argument, a subsequence,
	still denoted by \(v_j\), satisfies
	\[
		v_j\longrightarrow z
		\qquad\text{in }C^4_{\mathrm{loc}}(\mathbb R),
	\]
	Passing to the limit in the equation and the bounds, we find that
	\[
		\mathcal Lz=z^{-7},
		\qquad c\leq z\leq C
		\quad\text{on }\mathbb R.
	\]
	The rigidity statement proved above implies that \(z\equiv c_*\).

	If \(v(t)\) did not converge to \(c_*\) as \(t\to-\infty\), one could
	choose \(t_j\to-\infty\) and \(\varepsilon>0\) with
	\(|v(t_j)-c_*|\geq\varepsilon\). The preceding compactness argument would
	produce a subsequence for which \(v(t_j+\cdot)\to c_*\) locally uniformly,
	contradicting the inequality at \(s=0\). Therefore
	\[
		\lim_{t\to-\infty}v(t)=c_*.
	\]
	Since \(r=e^t\) and \(v(t)=u(r)/\sqrt r\), this is precisely
	\eqref{eq:radial-origin-profile}.
\end{proof}
	
	\begin{proof}[Proof of Theorem~\ref{thm1}]
		By Lemma~\ref{lem:radial-symmetry-vanishing-origin}, \(u\) is radially
		symmetric about the origin. Lemma~\ref{lem:radial-origin-profile} then yields
		\begin{equation*}
			\lim_{r\downarrow0}\frac{u(r)}{\sqrt r}
			=\left(\frac43\right)^{1/4},
		\end{equation*}
		which is \eqref{eq:main-origin-asymptotic}.

		It remains to prove completeness at the puncture. The preceding limit
		implies that there are \(r_0,C>0\) such that
		\begin{equation*}
			u(x)\leq C|x|^{1/2}
			\qquad(0<|x|<r_0).
		\end{equation*}
		Let \(s_0\in(0,\infty]\), and let
		\(\gamma:[0,s_0)\to\mathbb R^3\setminus\{0\}\) be a locally piecewise
		\(C^1\) curve such that \(|\gamma(s)|\to0\) as \(s\uparrow s_0\).
		Choose \(s_1\in[0,s_0)\) so that
		\(0<|\gamma(s)|<r_0\) for every \(s\in[s_1,s_0)\). The line element of
		\(g=u^{-4}|dx|^2\) is \(ds_g=u^{-2}|dx|\); hence, for
		\(s\in(s_1,s_0)\),
		\begin{equation*}
			\begin{aligned}
				L_g\bigl(\gamma|_{[s_1,s]}\bigr)
				&=\int_{s_1}^{s}u(\gamma(\tau))^{-2}|\gamma'(\tau)|\,d\tau\\
				&\geq C^{-2}\int_{s_1}^{s}
				\frac{|\gamma'(\tau)|}{|\gamma(\tau)|}\,d\tau\\
				&\geq C^{-2}\int_{s_1}^{s}
				\left|\frac{d}{d\tau}\log|\gamma(\tau)|\right|\,d\tau\\
				&\geq C^{-2}
				\left|\log|\gamma(s)|-\log|\gamma(s_1)|\right|.
			\end{aligned}
		\end{equation*}
		Here we used
		\[
			\left|\frac{d}{d\tau}\log|\gamma(\tau)|\right|
			=\frac{|\gamma(\tau)\cdot\gamma'(\tau)|}{|\gamma(\tau)|^2}
			\leq\frac{|\gamma'(\tau)|}{|\gamma(\tau)|}
		\]
		at every differentiability point of \(\gamma\).
		Since \(\log|\gamma(s)|\to-\infty\) as \(s\uparrow s_0\), every such
		curve has infinite \(g\)-length. Thus \(g\) is complete at the origin.
	\end{proof}
	
\section{Rigidity under sublinear growth}\label{sec:rigidity}

In this section, we prove Theorem~\ref{thm2} through a weighted uniqueness argument for the corresponding equation on the cylinder. Set \(\mathcal C=\mathbb R\times\mathbb S^2\) and define
\begin{equation}\label{eq:apsc-two-solution-operator}
	P=\partial_t^4+
	\left(2\Delta_{\mathbb S^2}-\frac52\right)\partial_t^2+
	\left(\Delta_{\mathbb S^2}+\frac34\right)^2.
\end{equation}
Under the Emden--Fowler transformation, Euclidean dilations correspond to translations in \(t\), while rotations act on the spherical variable. The following lemma establishes uniqueness for positive cylindrical solutions whose difference is \(o(\cosh(t/2))\) at both ends. The rupture condition at the origin and the sublinear growth assumption at infinity provide the required decay at the two ends of the cylinder. The uniqueness lemma can therefore be applied to the cylindrical solution and each of its translates and rotations.

\begin{lem}\label{lem:apsc-two-solution-uniqueness}
	Suppose \(V_1,V_2\in C^4(\mathcal C)\) are positive solutions of
	\[
	PV_i=V_i^{-7}\qquad\text{on }\mathcal C.
	\]
	If
	\begin{equation}\label{eq:apsc-two-solution-decay}
		\lim_{t\to\pm\infty}
		\sup_{\theta\in\mathbb S^2}
		\frac{|V_1(t,\theta)-V_2(t,\theta)|}{\cosh(t/2)}=0,
	\end{equation}
	then \(V_1\equiv V_2\).
\end{lem}

\begin{proof}
	Set
	\[
	h(t)=\cosh(t/2),\qquad T(t)=\tanh(t/2),\qquad
	Z=h^{-1}(V_1-V_2).
	\]
	By \eqref{eq:apsc-two-solution-decay},
	\begin{equation}\label{eq:apsc-two-solution-Z-decay}
		\lim_{t\to\pm\infty}\sup_{\theta\in\mathbb S^2}|Z(t,\theta)|=0.
	\end{equation}
Subtracting \(PV_2=V_2^{-7}\) from \(PV_1=V_1^{-7}\) and applying the mean-value formula, we obtain $$P(V_1-V_2)+a(V_1-V_2)=0,$$
 where
$$ a(t,\theta) =7\int_0^1 \bigl((1-\sigma)V_2(t,\theta)+\sigma V_1(t,\theta)\bigr)^{-8} \,d\sigma>0. $$ Using the substitution $V_1 - V_2 = hZ$ with $h > 0$, we define $\mathcal{B}Z := h^{-1}P(hZ)$. Since \(h'/h=T/2\) and \(h''/h=1/4\), by the Leibniz rule, we have
$$ h^{-1}\partial_t^2(hZ) =Z_{tt}+TZ_t+\frac14Z $$
and
$$ h^{-1}\partial_t^4(hZ) =Z_{tttt}+2TZ_{ttt} +\frac32Z_{tt}+\frac T2Z_t+\frac1{16}Z. $$
Consequently,
\begin{equation*}\label{eq:apsc-two-solution-conjugation}
	\begin{aligned}
	\mathcal BZ&=h^{-1}P(hZ)=h^{-1}\left(\partial_t^4+
	\left(2\Delta_{\mathbb S^2}-\frac52\right)\partial_t^2+
	\left(\Delta_{\mathbb S^2}+\frac34\right)^2\right)(hZ)\\
	&=Z_{tttt}+2TZ_{ttt}
	+\frac32Z_{tt}+\frac T2Z_t+\frac1{16}Z+\left(2\Delta_{\mathbb S^2}-\frac52\right)
	\left(Z_{tt}+TZ_t+\frac14Z\right)
	+\left(\Delta_{\mathbb S^2}+\frac34\right)^2Z\\
&=Z_{tttt}+2TZ_{ttt}
		+(2\Delta_{\mathbb S^2}-1)Z_{tt}+2T(\Delta_{\mathbb S^2}-1)Z_t
		+\Delta_{\mathbb S^2}(\Delta_{\mathbb S^2}+2)Z.
	\end{aligned}
\end{equation*}
A direct calculation shows that $Z$ satisfies
\begin{equation}\label{eq:apsc-two-solution-equation}
	\mathcal{B}Z + aZ = 0.
\end{equation}

To prove \(Z\equiv0\), we combine a global coercive identity with an
exhaustion of \(\mathcal C\). Since \(a\) need not be globally bounded, we
first establish a Caccioppoli estimate, uniform in both the cylinder center
and \(a\), to control the cutoff errors. The endpoint decay
\eqref{eq:apsc-two-solution-Z-decay} will then allow us to remove the cutoff.

	\medskip
	\emph{A uniform Caccioppoli estimate.} For \(s\in\mathbb R\) and
	\(b>0\), let
	\(\mathcal C_s^b=(s-b,s+b)\times\mathbb S^2\). Choose
	\(\zeta=\zeta(t)\in C_c^\infty(\mathbb R)\) such that
	\(0\leq\zeta\leq1\), \(\zeta\equiv1\) on \((s-1,s+1)\),
	\(\operatorname{supp}\zeta\subset(s-2,s+2)\), and its derivatives are bounded independently of \(s\). Let
	\(\mathcal D=\partial_t^2+\Delta_{\mathbb S^2}\), \(q=\zeta^4\), and set
	\[
	A = \int_{\mathcal{C}} q|\mathcal{D}Z|^2\,dtd\theta, \qquad
	J = \int_{\mathcal{C}} \zeta^2|\nabla_{\mathcal{C}}Z|^2\,dtd\theta, \qquad
	M = \int_{\mathcal{C}_s^2} Z^2\,dtd\theta.
	\]
	Multiplying \eqref{eq:apsc-two-solution-equation} by \(qZ\) and integrating over the cylinder $\mathcal{C}$, we obtain 
	\begin{equation}\label{eq:core_identity}
		\int_{\mathcal{C}} qZ(\mathcal{B}Z)\,dtd\theta+ \int_{\mathcal{C}} aqZ^2\,dtd\theta = 0.
	\end{equation}
	Observe that \(\mathcal{B} = \mathcal{D}^2 + 2T\partial_t\mathcal{D} - \partial_t^2 + 2\Delta_{\mathbb{S}^2} - 2T\partial_t\), then
	\begin{equation}\label{eq:expansion}
		\int_{\mathcal{C}} qZ(\mathcal{B}Z)\,dtd\theta= \int_{\mathcal{C}} qZ(\mathcal{D}^2 Z + 2T\partial_t\mathcal{D}Z)\,dtd\theta + \int_{\mathcal{C}} qZ(-\partial_t^2 Z + 2\Delta_{\mathbb{S}^2} Z - 2T\partial_t Z)\,dtd\theta.
	\end{equation}
	Integrating by parts twice, we obtain 
	\begin{equation}
		\begin{aligned}
		\int_{\mathcal C} qZ(\mathcal D^2Z+2T\partial_t\mathcal DZ)\,dtd\theta
		&=\int_{\mathcal C}\mathcal D(qZ)\mathcal DZ\,dtd\theta-2\int_{\mathcal C}\partial_t(qTZ)\mathcal DZ\,dtd\theta\\
		&=A+\int_{\mathcal C}
		\bigl[2q'Z_t+q''Z-2qTZ_t-2(qT)'Z\bigr]\mathcal DZ\,dtd\theta.
		\label{eq:apsc-two-solution-local-high}
		\end{aligned}
	\end{equation}
	Similarly, 
	\begin{align*}
		\int_{\mathcal C} qZ(-Z_{tt}+2\Delta_{\mathbb S^2}Z-2TZ_t)\,dtd\theta
		=\int_{\mathcal C} q\bigl(Z_t^2-2|\nabla_{\mathbb S^2}Z|^2\bigr)\,dtd\theta+\int_{\mathcal C}\biggl[(qT)'-\frac12q''\biggr]Z^2\,dtd\theta.
	\end{align*}
	Since \(|q'|\leq C\zeta^3\) and
	\(|q''|+|(qT)'|\leq C\zeta^2\), Young's inequality applied to
	\eqref{eq:apsc-two-solution-local-high} gives, for every
	\(\varepsilon>0\),
	\begin{equation}\label{eq:apsc-two-solution-local-error}
		\left|\int_{\mathcal C}
		\bigl[2q'Z_t+q''Z-2qTZ_t-2(qT)'Z\bigr]\mathcal DZ\,dtd\theta\right|
		\le\varepsilon A+C_\varepsilon J+C_\varepsilon M.
	\end{equation}
   Likewise,
	\begin{equation}\label{eq:apsc-two-solution-local-low-bound}
		\left|\int_{\mathcal C}qZ
		(-Z_{tt}+2\Delta_{\mathbb S^2}Z-2TZ_t)\,dtd\theta\right|
		\leq CJ+CM.
	\end{equation}
	Combining \eqref{eq:expansion} with \eqref{eq:apsc-two-solution-local-error}--\eqref{eq:apsc-two-solution-local-low-bound}, and substituting the result into \eqref{eq:core_identity}, we choose $\varepsilon=1/2$ to obtain
	\begin{equation}\label{eq:apsc-two-solution-A-estimate}
		A+\int_{\mathcal C} aqZ^2\,dtd\theta\le CJ+CM.
	\end{equation}
	Since $\zeta$ depends only on $t$ and has compact support, integration by parts in \(t\) and on \(\mathbb S^2\) yields 
	\begin{align*}
		J&=\int_{\mathcal C}\zeta^2(Z_t^2+|\nabla_{\mathbb S^2}Z|^2)\,dtd\theta=-\int_{\mathcal C}Z\,\partial_t(\zeta^2Z_t)\,dtd\theta
		-\int_{\mathcal C}\zeta^2Z\Delta_{\mathbb S^2}Z\,dtd\theta\\
		&=-\int_{\mathcal C}\zeta^2Z\mathcal DZ\,dtd\theta
		-2\int_{\mathcal C}\zeta\zeta'ZZ_t\,dtd\theta\le \varepsilon A+\frac12 J+C_\varepsilon M.
	\end{align*}
	Thus \(J\leq2\varepsilon A+C_\varepsilon M\). Combining this with \eqref{eq:apsc-two-solution-A-estimate} and choosing \(\varepsilon>0\)
	so that \(2C\varepsilon\leq1/2\), we obtain
	\(A+\int_{\mathcal C}aqZ^2\,dtd\theta\leq CM\), and then
	\begin{equation}\label{eq:apsc-two-solution-AJ}
		A+J+\int_{\mathcal C}aqZ^2\,dtd\theta\le CM.
	\end{equation}
	Expanding \(A\) and using \(|q''|\leq C\zeta^2\), we have
	\begin{equation*}
		\int_{\mathcal C}q
		\left[
		Z_{tt}^2
		+2|\nabla_{\mathbb S^2}Z_t|^2
		+(\Delta_{\mathbb S^2}Z)^2
		\right]\,dtd\theta=A+\int_{\mathcal C}
		q''|\nabla_{\mathbb S^2}Z|^2\,dtd\theta\le A+CJ.
	\end{equation*}
	It follows that 
	\begin{align*}
		\int_{\mathcal C_s^1}
		\left(
		Z_{tt}^2+|\nabla_{\mathcal C}Z|^2+aZ^2
		\right)\,dtd\theta&\leq
		\int_{\mathcal C}q
		\left[
		Z_{tt}^2
		+2|\nabla_{\mathbb S^2}Z_t|^2
		+(\Delta_{\mathbb S^2}Z)^2
		\right]\,dtd\theta
		+J+\int_{\mathcal C}aqZ^2\,dtd\theta\\
		&\leq
		A+CJ+\int_{\mathcal C}aqZ^2\,dtd\theta.
	\end{align*}
	Together with
	\eqref{eq:apsc-two-solution-AJ}, we get
	\begin{equation}\label{eq:apsc-two-solution-caccioppoli}
		\int_{\mathcal C_s^1}
		\left(
		Z_{tt}^2+|\nabla_{\mathcal C}Z|^2+aZ^2
		\right)\,dtd\theta
		\leq C\int_{\mathcal C_s^2}Z^2\,dtd\theta,
	\end{equation}
	where \(C\) is independent of \(s\) and \(a\).
	
	\medskip
	\emph{A global coercive identity.} Choose a nonnegative
	\(\eta\in C_c^\infty(\mathbb R)\). Testing
	\eqref{eq:apsc-two-solution-equation} against \(\eta Z\) and integrating
	by parts gives
	\begin{equation}\label{eq:apsc-two-solution-global-identity}
		Q_\eta[Z]+\int_{\mathcal C}\eta aZ^2\,dtd\theta=-\mathcal R_\eta[Z],
	\end{equation}
	where
	\begin{align}
		Q_\eta[Z]=\int_{\mathcal C}\eta\bigl[&
		Z_{tt}^2+2|\nabla_{\mathbb S^2}Z_t|^2+(1+3T')Z_t^2\notag\\
		&+(\Delta_{\mathbb S^2}Z)^2-2|\nabla_{\mathbb S^2}Z|^2
		+T'|\nabla_{\mathbb S^2}Z|^2+(T'-T''')Z^2\bigr]\,dtd\theta,
		\label{eq:apsc-two-solution-Q}
	\end{align}
	and 
	\begin{align}
		\mathcal R_\eta[Z]={}&
		\int_{\mathcal C}\eta''
		\left(ZZ_{tt}-Z_t^2-\frac12Z^2-|\nabla_{\mathbb S^2}Z|^2\right)\,dtd\theta\notag\\
		&+\int_{\mathcal C}\eta'T
		\left(3Z_t^2+|\nabla_{\mathbb S^2}Z|^2\right)\,dtd\theta\notag\\
		&+\int_{\mathcal C}
		\left(\eta'T-\eta'''T-3\eta''T'-3\eta'T''\right)Z^2\,dtd\theta.
		\label{eq:apsc-two-solution-remainder}
	\end{align}
	For each fixed $t$, we can decompose $Z(t, \theta)$ into spherical harmonics as
	\[
	Z(t,\theta) = \sum_{\ell=0}^{\infty}\sum_k
	Z_{\ell,k}(t)Y_{\ell,k}(\theta),
	\]
	where $\{Y_{\ell,k}\}$ is an orthonormal basis of $L^2(\mathbb{S}^2)$ composed of eigenfunctions of $-\Delta_{\mathbb{S}^2}$ satisfying $-\Delta_{\mathbb S^2}Y_{\ell,k} = \lambda_\ell Y_{\ell,k}$ with $\lambda_\ell = \ell(\ell+1)$.
	By Parseval's identity, we get
	\[
	\int_{\mathbb S^2}
	(\Delta_{\mathbb S^2}Z)^2\,d\theta
	=\sum_{\ell,k}\lambda_\ell^2
	|Z_{\ell,k}(t)|^2.
	\]
	Moreover, integration by parts on \(\mathbb S^2\) yields
	\[
	\int_{\mathbb S^2}
	|\nabla_{\mathbb S^2}Z|^2\,d\theta
	=-\int_{\mathbb S^2}
	Z\Delta_{\mathbb S^2}Z\,d\theta
	=\sum_{\ell,k}\lambda_\ell
	|Z_{\ell,k}(t)|^2.
	\]
	Consequently,
	\[
	\int_{\mathbb S^2}
	\left[
	(\Delta_{\mathbb S^2}Z)^2
	-2|\nabla_{\mathbb S^2}Z|^2
	\right]\,d\theta
	=
	\sum_{\ell,k}
	\lambda_\ell(\lambda_\ell-2)
	|Z_{\ell,k}(t)|^2
	\geq0.
	\]
	Moreover,
	\[
	T'=\frac12\operatorname{sech}^2(t/2)>0,
	\qquad
	T'-T'''=\frac34\operatorname{sech}^4(t/2)>0.
	\]
	It follows from \eqref{eq:apsc-two-solution-Q} that
	\begin{equation}\label{eq:apsc-two-solution-coercivity}
		Q_\eta[Z]\ge
		\frac34\int_{\mathcal C}
		\eta\operatorname{sech}^4(t/2)Z^2\,dtd\theta\ge0.
	\end{equation}
	
	\medskip
	\emph{Removal of the cutoff.} Choose \(\eta_0\in C_c^\infty(\mathbb R)\) such that \(0\leq\eta_0\leq1\), \(\eta_0\equiv1\) on \([-1,1]\), and \(\operatorname{supp}\eta_0\subset[-2,2]\). For $R \ge 5$, set $\eta_R(t) = \eta_0(t/R)$ and $\mathcal{A}_R = \{(t, \theta) : R \le \vert{}t\vert{} \le 2R\}$. Then \(\operatorname{supp}\eta_R^{(k)}\times\mathbb S^2
	\subset\mathcal A_R\) and \( \|\eta_R^{(k)}\|_\infty\le C_kR^{-k}\), \(1\le k\le 3\). Since \(T,T'\), and \(T''\) are bounded, it follows from
	\eqref{eq:apsc-two-solution-remainder} that
	\begin{align}
		|\mathcal R_{\eta_R}[Z]|
		&\leq\frac{C}{R}\int_{\mathcal A_R}
		\bigl(Z^2+Z_t^2+|\nabla_{\mathbb S^2}Z|^2+|ZZ_{tt}|\bigr)\,dtd\theta\notag\\
		&\leq\frac{C}{R}\int_{\mathcal A_R}
		\bigl(Z^2+Z_{tt}^2+|\nabla_{\mathcal C}Z|^2\bigr)\,dtd\theta,
		\label{eq:apsc-two-solution-remainder-bound}
	\end{align}
	where the last inequality follows from Young's inequality.
	Define
	\[
	\omega(\rho)=
	\sup_{\{|t|\ge\rho\}\times\mathbb S^2}|Z(t,\theta)|.
	\]
	From \eqref{eq:apsc-two-solution-Z-decay}, \(\omega(\rho)\to0\) as
	\(\rho\to\infty\). Choose \(s_1,\ldots,s_{N_R}\), with \(N_R\leq CR\), such that \(\mathcal A_R\subset\bigcup_{k=1}^{N_R}\mathcal C_{s_k}^1\) and \(\mathcal C_{s_k}^2 \subset\{|t|\geq R-4\}\times\mathbb S^2\). By \eqref{eq:apsc-two-solution-caccioppoli} and \eqref{eq:apsc-two-solution-remainder-bound}, as \(R\to\infty\), we have 
	\begin{align*}
	|\mathcal R_{\eta_R}[Z]|\le\frac{C}{R}\int_{\mathcal A_R}
		\left(Z^2+Z_{tt}^2+|\nabla_{\mathcal C}Z|^2\right)\,dtd\theta\leq\frac{C}{R}\sum_{k=1}^{N_R}
		\int_{\mathcal C_{s_k}^2}Z^2\,dtd\theta\leq C\omega(R-4)^2\to 0.
	\end{align*}

	Fix \(L>0\). Then \(\eta_R\equiv1\) on \([-L,L]\) for every
	sufficiently large \(R\). Since \(a>0\), \eqref{eq:apsc-two-solution-global-identity} and \eqref{eq:apsc-two-solution-coercivity} yield
	\[
	0\le
	\frac34\int_{[-L,L]\times\mathbb S^2}
	\operatorname{sech}^4(t/2)Z^2
	\leq|\mathcal R_{\eta_R}[Z]|\longrightarrow0
	\quad\text{as }R\to\infty.
	\]
	Since \(Z\) is continuous, it follows that \(Z=0\) on
	\([-L,L]\times\mathbb S^2\). As \(L>0\) is arbitrary,
	\(Z\equiv0\) on \(\mathcal C\), and hence \(V_1\equiv V_2\).
\end{proof}

\begin{proof}[Proof of Theorem~\ref{thm2}]
	Let \(x=e^t\theta\) and set \( V(t,\theta)=e^{-t/2}u(e^t\theta)\).
	Since \((t,\theta)\mapsto e^t\theta\) is a smooth diffeomorphism from
	\(\mathcal C\) onto \(\mathbb R^3\setminus\{0\}\), we have
	\(V\in C^4(\mathcal C)\) and \(V>0\).
	The polar-coordinate identity
	\[
	\Delta\!\left(r^\alpha F(\log r,\theta)\right)
	=r^{\alpha-2}
	\left[\partial_t^2+(2\alpha+1)\partial_t+
	\Delta_{\mathbb S^2}+\alpha(\alpha+1)\right]F
	\]
	applied with \(\alpha=1/2\) and then \(\alpha=-3/2\) shows that
	\[
	PV=V^{-7}\qquad\text{on }\mathcal C,
	\]
	with \(P\) given by \eqref{eq:apsc-two-solution-operator}.
	
	For \(\tau\in\mathbb R\) and \(\mathcal O\in \mathrm{SO}(3)\), define
	\[
	V_{\tau,\mathcal O}(t,\theta)=V(t+\tau,\mathcal O\theta).
	\]
	Then \(V_{\tau,\mathcal O}\in C^4(\mathcal C)\) is positive. Since \(P\)
	commutes with translations in \(t\) and rotations on \(\mathbb S^2\),
	\[
	PV_{\tau,\mathcal O}(t,\theta)
	=(PV)(t+\tau,\mathcal O\theta)
	=V_{\tau,\mathcal O}(t,\theta)^{-7}.
	\]
	For each fixed \(\tau\in\mathbb R\) and
	\(\mathcal O\in\mathrm{SO}(3)\), 
	\begin{align*}
		\sup_{\theta\in\mathbb S^2}
		\frac{|V_{\tau,\mathcal O}(t,\theta)-V(t,\theta)|}
		{\cosh(t/2)}\leq\frac{2}{1+e^t}
		\left(
		e^{-\tau/2}\sup_{\theta\in\mathbb S^2}u(e^{t+\tau}\theta)
		+\sup_{\theta\in\mathbb S^2}u(e^t\theta)
		\right)\rightarrow0
		\quad\text{as }t\to-\infty,
	\end{align*}
	since \(u(x)\to0\) as \(x\to0\). By \eqref{eq:apsc-infinity-assumption}, we have
	\begin{align*}
		\sup_{\theta\in\mathbb S^2}
		\frac{|V_{\tau,\mathcal O}(t,\theta)-V(t,\theta)|}
		{\cosh(t/2)}\leq2\frac{e^t}{1+e^t}
		\left(
		e^{\tau/2}\sup_{\theta\in\mathbb S^2}
		\frac{u(e^{t+\tau}\theta)}{e^{t+\tau}}
		+\sup_{\theta\in\mathbb S^2}\frac{u(e^t\theta)}{e^t}
		\right)\rightarrow0
		\quad\text{as }t\to+\infty.
	\end{align*}
	It follows from Lemma~\ref{lem:apsc-two-solution-uniqueness} that
	\begin{equation}\label{eq:apsc-all-symmetries}
		V(t+\tau,\mathcal O\theta)=V(t,\theta),
		\qquad\tau\in\mathbb R,\ \mathcal O\in \mathrm{SO}(3).
	\end{equation}
	Taking \(\mathcal O=I\) in \eqref{eq:apsc-all-symmetries} shows that \(V\) is
	independent of \(t\); taking \(\tau=0\) and using the transitivity of
	\(\mathrm{SO}(3)\) on \(\mathbb S^2\) shows that it is independent of \(\theta\).
	Thus \(V\equiv c>0\). Substitution into \(PV=V^{-7}\) yields \(c=c_*\). It follows that
	\(u(x)=c_*|x|^{1/2}\).
	
	Finally,
	\[
	\Delta(c_*r^{1/2})=\frac34c_*r^{-3/2},
	\qquad
	|\nabla(c_*r^{1/2})|^2=\frac14c_*^2r^{-1}.
	\]
	The conformal scalar-curvature identity
	\eqref{eq:scalar-curvature-identity} therefore yields
	\(R_g=2c_*^4=8/3\).
\end{proof}

\section{Variational constructions and sharpness examples}

In this section, we use a variational method to construct two examples showing that the assumptions in Theorems~\ref{thm1} and~\ref{thm2} are essential. The first is a nonradial positive solution with sign-changing scalar curvature. The second is a family of radial positive solutions satisfying \(u_\ell(x)/|x|\to\ell>0\) as \(|x|\to\infty\). These examples show, respectively, that the curvature assumption in Theorem~\ref{thm1} cannot be omitted and that the condition \(u(x)=o(|x|)\) in Theorem~\ref{thm2} cannot be replaced by \(u(x)=O(|x|)\).

\begin{lem}\label{prop:variational-nonradial-counterexample}
	Set $c_*=(4/3)^{1/4}$ and $Q(x)=x_1^2+x_2^2+\frac32x_3^2$.
	There exists a positive solution
	$u\in C^\infty(\mathbb R^3\setminus\{0\})$, not radially symmetric about
	any point, of
	\begin{equation}\label{eq:vnc-equation}
		\Delta^2u=u^{-7}\qquad\text{in }\mathbb R^3\setminus\{0\}
	\end{equation}
	such that
	\begin{equation}\label{eq:vnc-asymptotics}
		\begin{aligned}
			u(x)&=c_*|x|^{1/2}+o\bigl(|x|^{1/2}\bigr)&&\text{as }x\to0,\\
			u(x)&=Q(x)+o\bigl(|x|^{1/2}\bigr)&&\text{as }|x|\to\infty,
		\end{aligned}
	\end{equation}
	with both remainders uniform in direction.
	The scalar curvature of $g=u^{-4}|dx|^2$ is positive in a punctured
	neighborhood of the origin and negative outside a sufficiently large ball.
\end{lem}
\begin{proof} We divide the proof into four steps.
	
	\emph{Step 1: The energy space.}
	Let \(\Omega=\mathbb R^3\setminus\{0\}\) and
	\(E=C_c^\infty(\Omega)\). On \(E\), set
	\[
	(\psi,\eta)_E
	=\int_{\mathbb R^3}\Delta\psi\,\Delta\eta\,dx,
	\qquad
	\|\psi\|_E=\|\Delta\psi\|_{L^2(\mathbb R^3)}.
	\]
	If \(\|\psi\|_E=0\), integration by parts gives
	\[
	0=-\int_{\mathbb R^3}\psi\Delta\psi\,dx
	=\int_{\mathbb R^3}|\nabla\psi|^2\,dx,
	\]
	so \(\psi\equiv0\). Thus \((\cdot,\cdot)_E\) is an inner product.
	Let \(\mathcal H\) be the Hilbert completion of \(E\). For any $w \in \mathcal{H}$, represented by an approximating Cauchy sequence $\{\psi_j\} \subset E$, its norm is
	\[
	\|w\|_{\mathcal H}
	=\lim_{j\to\infty}\|\Delta\psi_j\|_2,
	\]
	which vanishes only for the zero element of the completion.
	
	For \(\psi\in C_c^\infty(\Omega)\), regarded as its smooth zero
	extension to \(\mathbb R^3\), and for \(x\ne y\), Fourier inversion gives
	\begin{equation*}
		\begin{aligned}
			|\psi(x)-\psi(y)|&=\frac{1}{(2\pi)^{3/2}}\left| \int_{\mathbb{R}^3} e^{iy \cdot \xi} \left( e^{i(x-y) \cdot \xi} - 1 \right) \hat{\psi}(\xi) \, d\xi\right|\\
			&\le \frac{1}{(2\pi)^{3/2}}\int_{\mathbb{R}^3} \left| e^{i(x-y) \cdot \xi} - 1 \right| |\hat{\psi}(\xi)| \, d\xi\\
			&=\frac{1}{(2\pi)^{3/2}} \int_{\mathbb{R}^3} \frac{\left| e^{i(x-y) \cdot \xi} - 1 \right|}{|\xi|^2} |\widehat{\Delta \psi}(\xi)| \, d\xi,
		\end{aligned}
	\end{equation*}
	where the last equality follows from
	\(\widehat{\Delta\psi}(\xi)=-|\xi|^2\widehat\psi(\xi)\). The
	Cauchy--Schwarz inequality and the change of variables
	\(\eta=|x-y|\xi\) yield
	\begin{equation}\label{eq:vnc-holder}
		\begin{aligned}
			|\psi(x) - \psi(y)| &\le C \|\Delta \psi\|_2
			\left( |x-y| \int_{\mathbb{R}^3} \frac{\left| e^{i \frac{x-y}{|x-y|} \cdot \eta} - 1 \right|^2}{|\eta|^4} \, d\eta \right)^{1/2}\\
			&=C \|\Delta \psi\|_2 |x - y|^{1/2}.
		\end{aligned}
	\end{equation}
	The integral above is finite, since its integrand is
	\(O(|\eta|^{-2})\) near the origin and \(O(|\eta|^{-4})\) at infinity.
	
	Let \(w\in\mathcal H\), and let \(\{\psi_j\}\subset E\) be a Cauchy sequence representing \(w\). For every compact set \(K\subset\mathbb R^3\),  taking $y=0$ in \eqref{eq:vnc-holder}  and using $\psi_j(0)=\psi_k(0)=0$, we obtain
	\begin{equation*}
		\sup_K|\psi_j-\psi_k|
		\leq C_K\|\Delta(\psi_j-\psi_k)\|_2.
	\end{equation*}
	Hence there exists \(\widetilde w\in C(\mathbb R^3)\) such that \(\psi_j\rightarrow\widetilde w\) in \(C_{\mathrm{loc}}(\mathbb R^3)\) as \(j\to\infty\). Moreover, since \(\psi_j(0)=0\) for every \(j\), \(\widetilde w(0)=0\). If \(\{\varphi_j\}\subset E\) is another Cauchy sequence representing \(w\), then
	\begin{equation*}
		\sup_K|\psi_j-\varphi_j|
		\leq C_K\|\Delta(\psi_j-\varphi_j)\|_2
		\rightarrow0\quad\text{as }j\to\infty.
	\end{equation*}
	Thus \(\varphi_j\to\widetilde w\) in \(C_{\mathrm{loc}}(\mathbb R^3)\) as \(j\to\infty\), so \(\widetilde w\) is independent of the representing sequence. Consequently, every \(w\in\mathcal H\) admits a unique continuous representative, still denoted by \(w\), satisfying \(w(0)=0\). Moreover, \(\{\Delta\psi_j\}\) converges to some
	\(q\in L^2(\mathbb R^3)\). For every
	\(\eta\in C_c^\infty(\mathbb R^3)\),
	\[
	\int_{\mathbb R^3}q\eta\,dx
	=\lim_{j\to\infty}\int_{\mathbb R^3}\Delta\psi_j\,\eta\,dx
	=\lim_{j\to\infty}\int_{\mathbb R^3}\psi_j\Delta\eta\,dx
	=\int_{\mathbb R^3}w\Delta\eta\,dx.
	\]
	Thus \(q=\Delta w\) in the sense of distributions. Consequently,
	\[
	\|w\|_{\mathcal H}=\|\Delta w\|_2,\qquad
	\langle w_1,w_2\rangle_{\mathcal H}
	=\int_{\mathbb R^3}\Delta w_1\Delta w_2\,dx.
	\]
	Passing to the limit in \eqref{eq:vnc-holder}, we obtain
	\begin{equation}\label{eq:vnc-point-evaluation}
		|w(x)-w(y)|
		\leq C\|w\|_{\mathcal H}|x-y|^{1/2},
		\qquad
		|w(x)|
		\leq C\|w\|_{\mathcal H}|x|^{1/2}.
	\end{equation}
	
	Finally, for each fixed \(j\), the function \(\psi_j\) vanishes near
	both zero and infinity. Applying \eqref{eq:vnc-point-evaluation} to
	\(w-\psi_j\) gives
	\[
	\max\left\{
	\limsup_{r\downarrow0}\sup_{|x|=r}
	\frac{|w(x)|}{r^{1/2}},
	\limsup_{r\to\infty}\sup_{|x|=r}
	\frac{|w(x)|}{r^{1/2}}
	\right\}
	\leq C\|w-\psi_j\|_{\mathcal H}.
	\]
	Letting \(j\to\infty\), we conclude that
	\begin{equation}\label{eq:vnc-endpoint-smallness}
		\lim_{r\downarrow0}\sup_{|x|=r}
		\frac{|w(x)|}{r^{1/2}}
		=
		\lim_{r\to\infty}\sup_{|x|=r}
		\frac{|w(x)|}{r^{1/2}}
		=0.
	\end{equation}
	
	\medskip
	\emph{Step 2: Variational framework.}
	Set \(s(x)=c_*|x|^{1/2}\), so that \(\Delta^2s=s^{-7}\) in
	\(\Omega\). Choose \(\chi\in C^\infty([0,\infty))\) such that
	\(0\leq\chi\leq1\), \(\chi=1\) on \([0,1]\), and \(\chi=0\) on
	\([2,\infty)\). Define
	\[
	\phi=\chi(|x|)s+(1-\chi(|x|))Q,
	\qquad
	f=\Delta^2\phi-\phi^{-7}.
	\]
	Then \(\phi>0\), \(f=0\) for \(0<|x|<1\), and \(f=-Q^{-7}\) for
	\(|x|>2\). Since \(Q(x)\geq|x|^2\),
	\begin{equation}\label{eq:vnc-residual}
		\int_\Omega |f(x)|\,|x|^{1/2}\,dx<\infty,
		\qquad
		\left|\int_\Omega fw\,dx\right|
		\leq C_f\|w\|_{\mathcal H}.
	\end{equation}
	Set
	\[
	\Phi(t)=
	\begin{cases}
		\frac16t^{-6},&t>0,\\
		+\infty,&t\leq0,
	\end{cases}
	\]
	and define
	\begin{equation}\label{G}
	G(x,a)=\Phi(\phi(x)+a)-\Phi(\phi(x))+\phi(x)^{-7}a.
	\end{equation}
	Convexity of \(\Phi\) gives \(G(x,a)\geq0\). Consider the energy functional
	\(\mathcal J:\mathcal H\to\mathbb R\cup\{+\infty\}\) defined by
	\begin{equation}\label{eq:vnc-functional}
		\mathcal J(w)
		=\frac12\|w\|_{\mathcal H}^2
		+\int_\Omega G(x,w(x))\,dx
		+\int_\Omega f(x)w(x)\,dx.
	\end{equation}
	By \eqref{eq:vnc-residual},
	\[
	\mathcal J(w)\geq
	\frac12\|w\|_{\mathcal H}^2-C_f\|w\|_{\mathcal H}.
	\]
	Thus \(\mathcal J\) is coercive and bounded from below. Let $\{w_j\} \subset \mathcal{H}$ be a minimizing sequence such that $\mathcal{J}(w_j) \to \inf_{\mathcal{H}} \mathcal{J}$ as \(j\to\infty\). Since $\mathcal J(0)=0$, \(\{w_j\}\) is bounded in \(\mathcal H\). Up to a subsequence, there exists \(w\in\mathcal H\) such that \(w_j\rightharpoonup w\) weakly in \(\mathcal H\) as \(j\to\infty\). For each fixed \(x\), \eqref{eq:vnc-point-evaluation} shows that
	\(v\mapsto v(x)\) is a bounded linear functional on \(\mathcal H\). Thus for any \(x\in\Omega\),
	\(w_j(x)\to w(x)\) as \(j\to\infty\). Since \(G(x,\cdot)\) is nonnegative and lower
	semicontinuous, Fatou's lemma gives
	\[
	\int_\Omega G(x,w(x))\,dx
	\leq\liminf_{j\to\infty}
	\int_\Omega G(x,w_j(x))\,dx.
	\]
	By \eqref{eq:vnc-residual}, the map
	\(v\mapsto\int_\Omega fv\,dx\) is a bounded linear functional on
	\(\mathcal H\), and is therefore weakly continuous. Together with the
	weak lower semicontinuity of the \(\mathcal H\)-norm, the preceding
	inequality gives
	\[
	\mathcal J(w)\leq\liminf_{j\to\infty}\mathcal J(w_j).
	\]
	Thus \(w\) is a minimizer of \(\mathcal J\). Since
	\(\mathcal J(w)\leq\mathcal J(0)=0\),
	\begin{equation}\label{eq:vnc-finite-potential}
		\int_\Omega G(x,w(x))\,dx<\infty.
	\end{equation}
	
	Set $u=\phi+w$. We claim that \(u>0\) in \(\Omega\). If
	\(u(x_0)<0\) at some point, then continuity gives an open set on which
	\(G(x,w)=+\infty\), contradicting \eqref{eq:vnc-finite-potential}. Hence
	\(u\geq0\). If \(u(x_0)=0\), choose \(B_\rho(x_0)\Subset\Omega\).
	Then \eqref{eq:vnc-point-evaluation} and the smoothness of \(\phi\) give
	\[
	0\leq u(x)=|u(x)-u(x_0)|\leq C|x-x_0|^{1/2}
	\qquad\text{near }x_0.
	\]
	Since \(\phi\) is bounded away from zero and \(w\) is bounded near \(x_0\),
	\[
	G(x,w)=\Phi(u)+O(1)\qquad\text{near }x_0.
	\]
	Consequently, for sufficiently small \(\rho>0\),
	\[
	\int_{B_\rho(x_0)}G(x,w)\,dx
	\geq c\int_{B_\rho(x_0)}|x-x_0|^{-3}\,dx-C=\infty,
	\]
	again contradicting \eqref{eq:vnc-finite-potential}. Thus \(u>0\) in
	\(\Omega\).
	
	Let \(\eta\in C_c^\infty(\Omega)\), set
	\(K=\operatorname{supp}\eta\), and let \(m=\min_Ku>0\). If
	\(0<|t|\|\eta\|_\infty\leq m/2\), then \(u+t\eta\geq m/2\) on \(K\).
	The mean-value theorem shows
	\[
	\left|\frac{\Phi(u+t\eta)-\Phi(u)}{t}\right|
	\leq C_K|\eta|.
	\]
	Since \(C_K|\eta|\in L^1(\Omega)\), we have 
	$$ 
	\begin{aligned}
		0=\left. \frac{d}{dt}\mathcal J(w+t\eta) \right|_{t=0}&= \int_\Omega \Delta w \Delta \eta \,dx + \int_\Omega f\eta\,dx+ \lim_{t \to 0}\int_\Omega \frac{G(x, w+t\eta) - G(x, w)}{t} \,dx\\
		&= \int_\Omega \Delta w \Delta \eta \,dx + \int_\Omega (f+\phi^{-7})\eta\,dx+ \lim_{t \to 0}\int_\Omega \frac{\Phi(u+t\eta) - \Phi(u)}{t} \,dx\\
		&= \int_\Omega \Delta w \Delta \eta \,dx + \int_\Omega (\phi^{-7} - u^{-7} + f)\eta \,dx\\
		&=\int_\Omega \Delta u \Delta \eta \,dx - \int_\Omega u^{-7}\eta \,dx.
	\end{aligned} 
	$$
	Here the third equality follows from the Dominated Convergence Theorem, and the last equality follows from the definition of $f$ and $u = \phi + w$. Hence \(u\) is a weak solution of
	\eqref{eq:vnc-equation}. Since \(u\) is continuous and strictly positive, \(u^{-7}\in L^p_{\mathrm{loc}}(\Omega)\) for every \(1<p<\infty\). Interior \(W^{4,p}\)-estimates yield
	\(u\in W_{\mathrm{loc}}^{4,p}(\Omega)\). Taking \(p>3\) and applying
	the Sobolev--Morrey embedding, we obtain
	\(u\in C_{\mathrm{loc}}^{3,\alpha}(\Omega)\) for some
	\(\alpha\in(0,1)\). Since the nonlinearity \(u^{-7}\) is smooth for \(u>0\), standard interior
	elliptic bootstrapping yields \(u\in C^\infty(\Omega)\).
	
	\medskip
	\emph{Step 3: Asymptotics and nonradiality.}
	By the definition of \(\phi\), we have \(\phi=s\) near zero, and \(\phi=Q\) near infinity. Hence
	\eqref{eq:vnc-endpoint-smallness} gives \eqref{eq:vnc-asymptotics}.
	If \(u\) were radial about some \(a\in\mathbb R^3\), then
	\(u(a+Re_1)=u(a+Re_3)\) for every sufficiently large \(R\). On the other hand, the uniform asymptotics at infinity yield
	\begin{align*}
		u(a+Re_3)-u(a+Re_1)
		&=Q(a+Re_3)-Q(a+Re_1)+o(R^{1/2})\\
		&=\frac12R^2+(3a_3-2a_1)R+o(R^{1/2}),
	\end{align*}
	which is positive for all sufficiently large \(R\). This is a contradiction, and so \(u\) is not radially symmetric about any point.
	
	\medskip
	\emph{Step 4: The curvature signs.}
	Let
	\[
	A=\bigl\{y:\tfrac12<|y|<2\bigr\},
	\qquad
	A'=\bigl\{y:\tfrac34<|y|<\tfrac32\bigr\},
	\]
	and define
	\[
	v_r(y)=r^{-1/2}u(ry),
	\qquad
	z_R(y)=R^{-2}u(Ry).
	\]
	Then
	\[
	\Delta^2v_r=v_r^{-7},
	\qquad
	\Delta^2z_R=R^{-12}z_R^{-7}.
	\]
	Since \(s\) and \(Q\) are homogeneous of degrees \(1/2\) and \(2\), respectively, for all sufficiently small \(r\) and sufficiently large \(R\),	
	$$
	v_r(y)-s(y)=r^{-1/2}w(ry),
	\qquad
	z_R(y)-Q(y)=R^{-2}w(Ry),
	\qquad y\in A.
	$$
It follows from	\eqref{eq:vnc-endpoint-smallness} that
	$$
	\lVert v_r-s\rVert_{L^\infty(A)}\rightarrow0
	\quad\text{as }r\to0^+,
	\qquad
	\lVert z_R-Q\rVert_{L^\infty(A)}\rightarrow0
	\quad\text{as }R\to\infty.
	$$
	Since \(s,Q>0\) on
	\(\overline A\), the functions \(v_r\) and \(z_R\) are uniformly
	bounded away from zero. Hence
	\[
	\begin{aligned}
		\|\Delta^2(v_r-s)\|_{L^\infty(A)}
		&=\|v_r^{-7}-s^{-7}\|_{L^\infty(A)}\rightarrow0\quad\text{as }r\to0^+,\\
		\|\Delta^2(z_R-Q)\|_{L^\infty(A)}
		&=R^{-12}\|z_R^{-7}\|_{L^\infty(A)}\rightarrow0\quad\text{as }R\to\infty.
	\end{aligned}
	\]
	For any fixed \(p>3\), the standard interior estimate
	\[
	\|h\|_{W^{4,p}(A')}
	\leq C\bigl(\|h\|_{L^p(A)}
	+\|\Delta^2h\|_{L^p(A)}\bigr)
	\]
	applied to \(h=v_r-s\) and \(h=z_R-Q\), followed by the
	Sobolev--Morrey embedding, yields
\begin{equation}\label{eq:vnc-rescaled-C2}
	v_r\rightarrow s\quad\text{in }C^2(\overline{A'})
	\quad(r\to0^+),
	\qquad
	z_R\rightarrow Q\quad\text{in }C^2(\overline{A'})
	\quad(R\to\infty).
\end{equation}
	
	For a positive function \(v\), set
	\[
	\mathcal R[v]
	=8v^2\bigl(v\Delta v-2|\nabla v|^2\bigr).
	\]
	Scaling gives
	\[
	\mathcal R[v_r](y)=R_g(ry),
	\qquad
	\mathcal R[z_R](y)=R^{-6}R_g(Ry).
	\]
	Since \(\mathcal R[s]=8/3\), the first convergence in
	\eqref{eq:vnc-rescaled-C2} implies
	\begin{equation*}
		\lim_{r\to0^+}\sup_{\theta\in\mathbb S^2}
		\left|R_g(r\theta)-\frac83\right|=0.
	\end{equation*}
	At infinity,
	\[
	\Delta Q=7,
	\qquad
	Q\Delta Q-2|\nabla Q|^2
	=-x_1^2-x_2^2-\frac{15}{2}x_3^2.
	\]
	Writing \(\theta=x/|x|\), a direct computation gives
	\[
	\mathcal R[Q](x)
	=-8|x|^6\left(1+\frac{\theta_3^2}{2}\right)^2
	\left(1+\frac{13\theta_3^2}{2}\right).
	\]
	The second convergence in \eqref{eq:vnc-rescaled-C2} therefore gives,
	uniformly on \(\mathbb S^2\),
	\begin{equation*}
		R^{-6}R_g(R\theta)
		\rightarrow
		\mathcal R[Q](\theta)
		=-8\left(1+\frac{\theta_3^2}{2}\right)^2
		\left(1+\frac{13\theta_3^2}{2}\right)
		\leq-8<0.
	\end{equation*}
	Thus \(R_g>0\) in a punctured neighborhood of \(0\) and \(R_g<0\)
	outside a sufficiently large ball. 
\end{proof}

\begin{cor}\label{cor:radial-linear-growth}
	For every \(\ell>0\), there exists a positive radial solution
	\(u_\ell\in C^\infty(\mathbb R^3\setminus\{0\})\) of \eqref{eq1} such that
	\begin{align*}\begin{aligned} u_\ell(x) &=c_*|x|^{1/2}+o\bigl(|x|^{1/2}\bigr) &&\text{as }x\to0,\\ u_\ell(x) &=\ell|x|+o\bigl(|x|^{1/2}\bigr) &&\text{as }|x|\to\infty. \end{aligned} \end{align*}
	In particular,	\(u_\ell\not\equiv u_*\) and \(u_\ell(x)/|x|\to\ell\) as \(|x|\to\infty\).
\end{cor}
\begin{proof}
	Set \(Q_\ell(x)=\ell|x|\). Then \(Q_\ell>0\) and
	\(\Delta^2Q_\ell=0\) in \(\Omega\). As in the proof of
	Lemma~\ref{prop:variational-nonradial-counterexample}, define
	$$
	\phi_\ell=\chi(|x|)s+(1-\chi(|x|))Q_\ell,
	\qquad
	f_\ell=\Delta^2\phi_\ell-\phi_\ell^{-7},
	$$
	and define \(G_\ell\) and \(\mathcal J_\ell\)  by \eqref{G} and \eqref{eq:vnc-functional}, with
	\((\phi,f)\) replaced by \((\phi_\ell,f_\ell)\). Since \(f_\ell=0\) for \(0<|x|<1\), \(f_\ell=-\ell^{-7}|x|^{-7}\) for \(|x|>2\), and \(f_\ell\) is smooth on \(1\leq|x|\leq2\), we have 
	   $$
	   \int_\Omega |f_\ell(x)|\,|x|^{1/2}\,dx<\infty.
	   $$
	   By the same argument as in the proof of Lemma~\ref{prop:variational-nonradial-counterexample}, \(\mathcal J_\ell\) admits a minimizer \(w_\ell\in\mathcal H\) such that \(u_\ell=\phi_\ell+w_\ell\) is positive and smooth in \(\Omega\) and satisfies \(\Delta^2 u_\ell=u_\ell^{-7}\) in \(\Omega\).

	It remains to prove that \(u_\ell\) is radial. Set \(\mathcal D_\ell
	=\{w\in\mathcal H:\mathcal J_\ell(w)<\infty\}\). Since \(\Phi\) is convex, the map \(a\mapsto G_\ell(x,a)\) is convex for every \(x\in\Omega\). For any \(w_1,w_2\in\mathcal D_\ell\) and \(0<\lambda<1\),
	$$
	G_\ell\bigl(x,(1-\lambda)w_1+\lambda w_2\bigr)
	\leq(1-\lambda)G_\ell(x,w_1)+\lambda G_\ell(x,w_2).
	$$
	Since the right-hand side is integrable, \((1-\lambda)w_1+\lambda w_2\in\mathcal D_\ell\). Thus
	\(\mathcal D_\ell\) is convex. Moreover
	$$
	\|(1-\lambda)w_1+\lambda w_2\|_{\mathcal H}^2
	=(1-\lambda)\|w_1\|_{\mathcal H}^2
	+\lambda\|w_2\|_{\mathcal H}^2
	-\lambda(1-\lambda)\|w_1-w_2\|_{\mathcal H}^2.
	$$
	Together with the convexity of \(G_\ell(x,\cdot)\), we obtain
	$$
	\begin{aligned}
		\mathcal J_\ell\bigl((1-\lambda)w_1+\lambda w_2\bigr)\leq(1-\lambda)\mathcal J_\ell(w_1)
		+\lambda\mathcal J_\ell(w_2)-\frac{\lambda(1-\lambda)}2
		\|w_1-w_2\|_{\mathcal H}^2.
	\end{aligned}
	$$
	Thus \(\mathcal J_\ell\) is strictly convex on \(\mathcal D_\ell\), and its minimizer \(w_\ell\) is unique.
	
	For \(\mathcal O\in\mathrm{SO}(3)\), define
	\(T_{\mathcal O}\psi=\psi\circ\mathcal O\) on \(E\). The rotational
	invariance of the Laplacian and a change of variables give
	\[
	\|T_{\mathcal O}\psi\|_E=\|\psi\|_E\qquad\text{for every }\psi\in E.
	\]
Since \(E\) is dense in \(\mathcal H\), \(T_{\mathcal O}\) extends uniquely to a linear isometry on \(\mathcal H\). Moreover, the radiality of \(\phi_\ell\) and \(f_\ell\) implies
	\[
	\mathcal J_\ell(T_{\mathcal O}w)
	=\mathcal J_\ell(w)
	\qquad\text{for every }w\in\mathcal H.
	\]
	Thus \(T_{\mathcal O}w_\ell\) is also a minimizer. By uniqueness, \(T_{\mathcal O}w_\ell=w_\ell\) for every \(\mathcal O\in\mathrm{SO}(3)\), and hence \(u_\ell=\phi_\ell+w_\ell\) is radial.
	
	Since \(\phi_\ell=s\) near the origin and
	\(\phi_\ell=Q_\ell\) near infinity, \eqref{eq:vnc-endpoint-smallness} yields
	\[
	\begin{aligned}
		u_\ell(x)
		&=c_*|x|^{1/2}+o\bigl(|x|^{1/2}\bigr)
		&&\text{as }x\to0,\\
		u_\ell(x)
		&=\ell|x|+o\bigl(|x|^{1/2}\bigr)
		&&\text{as }|x|\to\infty.
	\end{aligned}
	\]
	Hence, \(u_\ell\) is a positive solution of \eqref{eq1} and \(u_\ell\not\equiv u_*\).
\end{proof}

\section*{Acknowledgements}

The first author is supported by the National Natural Science Foundation of
China (NSFC, No.~12671132) and the Science and Technology Commission of
Shanghai Municipality (No.~22DZ2229014).


\begin{thebibliography}{99}
\bibitem{AxlerBourdonRamey2001}
S.~Axler, P.~Bourdon and W.~Ramey,
\emph{Harmonic Function Theory},
Second Edition, Graduate Texts in Mathematics, 137,
Springer-Verlag, New York, 2001.

\bibitem{Branson}
T.~P. Branson,
\emph{Differential operators canonically associated to conformal structures},
Math. Scand. \textbf{57} (1985), 293-345.

\bibitem{CGS}
L.~A. Caffarelli, B. Gidas and J. Spruck,
\emph{Asymptotic symmetry and local behavior of semilinear elliptic equations
with critical Sobolev growth},
Comm. Pure Appl. Math. \textbf{42} (1989), 271-297.

\bibitem{ChangYang}
S.-Y.~A. Chang and P.~C. Yang,
\emph{Extremal metrics of zeta functional determinants on $4$-manifolds},
Ann. of Math. (2) \textbf{142} (1995), 171-212.

\bibitem{ChoiXu}
Y.~S. Choi and X. W. Xu,
\emph{Nonlinear biharmonic equations with negative exponents},
J. Differential Equations \textbf{246} (2009), 216-234.

\bibitem{FrankKonig}
R.~L. Frank and T. K\"onig,
\emph{Classification of positive singular solutions to a nonlinear
biharmonic equation with critical exponent},
Anal. PDE \textbf{12} (2019), 1101-1113.


\bibitem{GazzolaGrunauSweers2010}
F. Gazzola, H.-C. Grunau and G. Sweers,
\emph{Polyharmonic Boundary Value Problems},
Lecture Notes in Mathematics 1991, Springer, Berlin, 2010.

\bibitem{GJMS}
C.~R. Graham, R. Jenne, L.~J. Mason and G.~A.~J. Sparling,
\emph{Conformally invariant powers of the Laplacian. I. Existence},
J. London Math. Soc. (2) \textbf{46} (1992), 557-565.



\bibitem{Guerra}
I. Guerra,
\emph{A note on nonlinear biharmonic equations with negative exponents},
J. Differential Equations \textbf{253} (2012), 3147-3157.

\bibitem{GuoHuangWangWei}
Z. M. Guo, X. Huang, L. P. Wang and J. C. Wei,
\emph{On Delaunay solutions of a biharmonic elliptic equation with critical
exponent}, J. Anal. Math. \textbf{140} (2020), 371-394.

\bibitem{GuoWeiZhou}
Z. M.  Guo, L. Wei, and F. Zhou,
\emph{Radial symmetry of positive entire solutions of a fourth order elliptic
equation with a singular nonlinearity},
arXiv:1804.08215 (2018).


\bibitem{HyderWei}
A. Hyder and J. Wei,
\emph{Non-radial solutions to a bi-harmonic equation with negative exponent},
Calc. Var. Partial Differential Equations \textbf{58} (2019), Paper No.~198.

\bibitem{KMPS}
N. Korevaar, R. Mazzeo, F. Pacard and R. Schoen,
\emph{Refined asymptotics for constant scalar curvature metrics with isolated
singularities},
Invent. Math. \textbf{135} (1999), 233-272.

\bibitem{Lin}
C.-S. Lin,
\emph{A classification of solutions of a conformally invariant fourth order
equation in $\RR^n$},
Comment. Math. Helv. \textbf{73} (1998), 206-231.

\bibitem{Marques}
F.~C. Marques,
\emph{Isolated singularities of solutions to the Yamabe equation},
Calc. Var. Partial Differential Equations \textbf{32} (2008), 349-371.

\bibitem{MazzeoPacard}
R. Mazzeo and F. Pacard,
\emph{Constant scalar curvature metrics with isolated singularities},
Duke Math. J. \textbf{99} (1999), 353-418.

\bibitem{McKennaReichel}
P.~J. McKenna and W. Reichel,
\emph{Radial solutions of singular nonlinear biharmonic equations and
applications to conformal geometry},
Electron. J. Differential Equations \textbf{2003} (2003), no.~37, 1-13.

\bibitem{NgoNguyenPhan}
Q.~A. Ng{\^o}, V.~H. Nguyen, and Q.~H. Phan,
\emph{A pointwise inequality for a biharmonic equation with negative exponent
and related problems},
Nonlinearity \textbf{31} (2018), no.~12, 5484-5499.

\bibitem{Paneitz}
S.~M. Paneitz,
\emph{A quartic conformally covariant differential operator for arbitrary
pseudo-Riemannian manifolds},
SIGMA Symmetry Integrability Geom. Methods Appl. \textbf{4} (2008),
Paper 036, 3 pp.

\bibitem{PeleskoBernstein}
J.~A. Pelesko and D.~H. Bernstein,
\emph{Modeling MEMS and NEMS},
Chapman \& Hall/CRC, Boca Raton, FL, 2003. xxiv+357 pp. ISBN: 1-58488-306-5.

%\bibitem{Santra}
%S. Santra,
%\emph{On the positive solutions for a perturbed negative exponent problem on
%$\RR^3$},
%Discrete Contin. Dyn. Syst. \textbf{38} (2018), 1441-1460.

\bibitem{WeiXu}
J. C. Wei and X. W. Xu,
\emph{Classification of solutions of higher order conformally invariant
equations},
Math. Ann. \textbf{313} (1999), no.2, 207-228.

\bibitem{XiongZhang}
J. G. Xiong and L. Zhang,
\emph{Isolated singularities of solutions to the Yamabe equation in dimension
$6$},
Int. Math. Res. Not. IMRN \textbf{2022} (2022), no.~12, 9571-9597.

\bibitem{Xu}
X. W. Xu,
\emph{Exact solutions of nonlinear conformally invariant integral equations
in $\RR^3$},
Adv. Math. \textbf{194} (2005), no.2,  485-503.

\bibitem{XuYang}
X. W. Xu and P.~C. Yang,
\emph{On a fourth order equation in 3-D}, ESAIM Control Optim. Calc. Var. \textbf{8} (2002), 1029-1042.

\end{thebibliography}
\end{document}